\documentclass[12pt]{article}
\usepackage{amsfonts}
\usepackage{amsmath}
\usepackage{amssymb}
\usepackage{geometry}
\usepackage{setspace}
\usepackage{sectsty}
\usepackage{lineno}
\usepackage{theorem}
\usepackage{fancyhdr}
\usepackage{graphicx}
\usepackage{xkeyval}

\newtheorem{theorem}{Theorem}

\newtheorem{corollary}[theorem]{Corollary}

\newtheorem{definition}[theorem]{Definition}

\newtheorem{lemma}[theorem]{Lemma}

\newtheorem{proposition}[theorem]{Proposition}
\newtheorem{remark}[theorem]{Remark}

\newenvironment{proof}[1][Proof]{\noindent\textbf{#1.} }{\ \rule{0.5em}{0.5em}}
\DeclareMathOperator{\var}{Var}

\DeclareMathOperator{\Span}{Span}

\input{tcilatex}

\begin{document}

\title{ On optimality of empirical risk minimization in linear aggregation}
\author{Adrien Saumard\thanks{%
Research partly supported by the french Agence Nationale de la Recherche
(ANR 2011 BS01 010 01 projet Calibration).} \\
CREST, Ensai, Universit\'{e} Bretagne Loire}
\maketitle

\begin{abstract}
In the first part of this paper, we show that the small-ball condition,
recently introduced by \cite{Mendel:15},\ may behave poorly for important
classes of localized functions such as wavelets, piecewise polynomials or
trigonometric polynomials, in particular leading to suboptimal estimates of
the rate of convergence of ERM for the linear aggregation problem. In a
second part, we recover optimal rates of covergence for the excess risk of
ERM when the dictionary is made of trigonometric functions. Considering the
bounded case, we derive the concentration of the excess risk around a single
point, which is an information far more precise than the rate of
convergence. In the general setting of a $L_{2}$ noise, we finally refine
the small ball argument by rightly selecting the directions we are looking
at, in such a way that we obtain optimal rates of aggregation for the
Fourier dictionary.

Keywords: empirical risk minimization, linear aggregation, small-ball
property, concentration inequality, empirical process theory.
\end{abstract}

\section{Introduction}

Consider the following general regression framework: $\left( \mathcal{X},%
\mathcal{T}_{\mathcal{X}}\right) $ is a measurable space, $\left( X,Y\right) 
\mathbb{\in }\mathcal{X\times \mathbb{R}}$ is a pair of random variables of
joint distribution $P$ - the marginal of $X$ being denoted $P^{X}$ - and it
holds%
\begin{equation}
Y=s_{\ast }\left( X\right) +\sigma \left( X\right) \varepsilon \text{ ,}
\label{reg_rel}
\end{equation}%
where $s_{\ast }$ is the regression function of the response variable $Y$
with respect to the \textit{random design} $X$, $\sigma \left( X\right) \geq
0$ is the \textit{heteroscedastic} \textit{noise} level and $\varepsilon $
is the conditionally standardized noise, satisfying $\mathbb{E}\left[
\varepsilon \left\vert X\right. \right] =0$ and $\mathbb{E}\left[
\varepsilon ^{2}\left\vert X\right. \right] =1$. Relation (\ref{reg_rel}) is
very general and is indeed satisfied as soon as $\mathbb{E}\left[ Y^{2}%
\right] <+\infty $. In this case $s_{\ast }\in L_{2}\left( P^{X}\right) $ is
the orthogonal projection of $Y$ onto the space of $X$-measurable functions.
In particular, no restriction is made on the structure of dependence between 
$Y$ and $X$.

We thus face a typical \textit{learning problem}, where the statistical
modelling is minimal, and the goal will be, given a sample $\left(
X_{i},Y_{i}\right) _{i=1}^{n}$ of law $P^{\otimes n}$ and a new covariate $%
X_{n+1}$, to predict the value of the associated response variable $Y_{n+1}$%
. More precisely, we want to construct a function $\hat{s}$, depending on
the data $\left( X_{i},Y_{i}\right) _{i=1}^{n}$, such that the least-squares
risk $R\left( \hat{s}\right) =\mathbb{E}\left[ \left( Y_{n+1}-\widehat{s}%
\left( X_{n+1}\right) \right) ^{2}\right] $ is as small as possible, the
pair $\left( X_{n+1},Y_{n+1}\right) $ being independent of the sample $%
\left( X_{i},Y_{i}\right) _{i=1}^{n}$.

In this paper, we focus on the technique of \textit{linear aggregation} via
Empirical Risk Minimization (ERM). This means that we are given a dictionary 
$S=\left\{ s_{1},...,s_{D}\right\} $ and that we produce the least-squares
estimator $\hat{s}_{m}$ on its linear span $m=%
\Span%
\left( S\right) $,%
\begin{equation}
\hat{s}_{m}\in \arg \min_{s\in m}R_{n}\left( s\right) \text{ , where }%
R_{n}\left( s\right) =\frac{1}{n}\sum_{i=1}^{n}\left( Y_{i}-s\left(
X_{i}\right) \right) ^{2}\text{ .}  \label{def_ERM}
\end{equation}%
The quantity $R_{n}\left( s\right) $ is called the empirical risk of the
function $s$. The accuracy of the method is tackled through an \textit{%
oracle inequality}, where the risk of the estimator $R\left( \hat{s}%
_{m}\right) $ is compared - on an event of probability close to one - to the
risk of the best possible function within the linear model $m$. The latter
function is denoted $s_{m}$ and is called the oracle, or the (orthogonal) 
\textit{projection} of the regression function $s_{\ast }$ onto $m$,%
\begin{equation*}
s_{m}\in \arg \min_{s\in m}R\left( s\right) \text{ .}
\end{equation*}%
An oracle inequality then writes, on an event $\Omega _{0}$ of probability
close to one,%
\begin{equation}
R\left( \hat{s}_{m}\right) \leq R\left( s_{m}\right) +r_{n}\left( D\right) 
\text{ ,}  \label{oracle_ineq}
\end{equation}%
for a positive residual term $r_{n}\left( D\right) $. An easy and classical
computation gives that the \textit{excess risk} satisfies $R\left( \hat{s}%
_{m}\right) -R\left( s_{m}\right) =\left\Vert \widehat{s}-s_{m}\right\Vert
_{2}^{2}$, where $\left\Vert \cdot \right\Vert _{2}$ is the natural
quadratic norm in $L_{2}\left( P^{X}\right) $, associated with the scalar
product $\left\langle f,g\right\rangle =\int f\left( x\right) g\left(
x\right) dP^{X}\left( x\right) .$ Hence, inequality (\ref{oracle_ineq}) can
be rewritten as $\left\Vert \hat{s}_{m}-s_{m}\right\Vert _{2}^{2}\leq
r_{n}\left( D\right) $ and the quantity $r_{n}\left( D\right) $ thus
corresponds to the rate of estimation of the projection $s_{m}$ by the
least-squares estimator $\hat{s}_{m}$ in terms of excess risk, corresponding
here to the squared quadratic norm.

The linear aggregation problem has been well studied in various settings
linked to nonparametric regression (%
\cite{Nem:00, tsybakov2003optimal, BunTsyWeg:07, AudCat:11}%
) and density estimation (\cite{RigTsy:07}). It has been consequently
understood that the \textit{optimal rate }$r_{n}\left( D\right) $ \textit{of
linear aggregation} is of the order of $D/n$, where $D$ is the size of the
dictionary. Recently, \cite{Lecue_Mendelson:15} have shown that ERM is
suboptimal for the linear aggregation problem in general, in the sense that
there exist a dictionary $S$ and a pair $\left( X,Y\right) $ of random
variables for which the rate of ERM (drastically) deteriorates, even in the
case where the response variable $Y$ and the dictionary are uniformly
bounded.

On the positive side, \cite{Lecue_Mendelson:15} also made a breakthrough by
showing that if a so-called \textit{small-ball condition} is achieved with
absolute constants, uniformly over the functions in the linear model $m$,
then the optimal rate is recovered by ERM. We recall and discuss in details
the small-ball condition in Section \ref{sec_small_ball}, but it is worth
mentioning here that one of the main advantages of the small-ball method
developed in a series of papers, 
\cite{Mendel:14, Mendel:15, Mendel:14a, KolMendel:15, LecMendel:14a, Lecue_Mendelson:15}
is that it enables to prove sharp bounds under very weak moment conditions
and thus to derive results that were \textit{unachievable} with more
standard concentration arguments.

In Section \ref{sec_small_ball}, we contribute to the growing understanding
of this very recent approach by looking at the behavior of the small-ball
condition when the dictionary is made of elements of some classical
orthonormal bases, such as histograms, piecewise polynomials, wavelets and
the Fourier basis. These examples are indeed central in various methods of
nonparametric statistics.

It appears that with such functions, the small-ball condition \textit{can't
be satisfied with absolute constants }and the resulting bounds obtained in 
\cite{Lecue_Mendelson:15} are far from optimal. This lack of accuracy of the
small-ball approach seems rather natural for dictionaries that are made of
localized functions, such as wavelets for instance, since these functions
are very "picky" and thus hardly identifiable - see Section \ref%
{sec_small_ball} for a more thorough discussion around these terms.

However, it seems more surprising that the Fourier dictionary also leads to
suboptimal rates of linear aggregation when analyzed \textit{via} the small
ball method. In fact, the behavior of the small-ball condition on the span
of some trigonometric functions is essentially unknown in the literature and
this type of information, in the related context of Fourier measurements in
compressed sensing, has a potentially significant impact on the theory of
Fourier measurements, \cite{LecMendel:14a}.

Nevertheless, we show in Section \ref{section_opt_bounds_Fourier} that ERM
achieves optimal rates of linear aggregation, both in the bounded setting
and for $L_{2}$-noise. Our result in particular outperform previously
obtained bounds \cite{AudCat:11}.

More precisely, when the variables are bounded, we derive \textit{%
concentration inequalities} for the excess risk, which is an information far
more precise than the rate of convergence. Our proofs are based on empirical
process theory and substantially simplify our previous approach to
concentration inequalities for the excess risk on models spanned by
localized bases, 
\cite{saum:12,NavSaum:16}%
.

When the noise is only assumed to have a second moment, we prove optimal
rates of linear aggregation for the Fourier dictionary by using a refined
small-ball argument. Indeed, by imposing a light and natural smoothness
condition on the regression function, we localize the analysis by only
looking at some directions in the model that satisfy a uniform small-ball
condition. It is important to note that such approach was suggested - but
not achieved - by Lecu\'{e} and Mendelson \cite{LecMendel:14a} for the study
of Fourier measurements in compressed sensing.

Finally, complete proofs are dispatched in Sections \ref%
{section_proofs_sec_2}\ and \ref{section_proof_fourier}, at the end of the
paper.

\section{The small-ball method for classical functional bases\label%
{sec_small_ball}}

We recall in Section \ref{ssection_rate}\ one of the main results of \cite%
{Lecue_Mendelson:15}, linking the small-ball condition to the rate of
convergence of ERM\ in linear aggregation. Then, we show in Section \ref%
{ssection_linear_bases} that the constants involved in the small-ball
condition behave poorly for dictionaries made of localized bases and also
for the Fourier basis.

\subsection{The small-ball condition and the rate of ERM in linear
aggregation\label{ssection_rate}}

Let us first recall the definition of the small-ball condition for a linear
span, as exposed in \cite{Lecue_Mendelson:15}.

\begin{definition}
\label{def_small_ball}A linear span $m\subset L_{2}\left( P^{X}\right) $ is
said to satisfy the \textbf{small-ball condition} for some positive
constants $\kappa _{0}$ and $\beta _{0}$ if for every $s\in m$,%
\begin{equation}
\mathbb{P}\left( \left\vert s\left( X\right) \right\vert \geq \kappa
_{0}\left\Vert s\right\Vert _{2}\right) \geq \beta _{0}\text{ .}
\label{ineq_small_ball}
\end{equation}
\end{definition}

The small-ball condition thus ensures that the functions of the model $m$ do
not put too much weight around zero. From a statistical perspective, it is
also explained in \cite{Lecue_Mendelson:15} that the small-ball condition
can be viewed as a quantified version of identifiability of the model $m$. A
more general small-ball condition - that reduces to the previous definition
for linear models - is also available when the model isn't necessary linear, 
\cite{Mendel:15}.

Under the small-ball condition, \cite{Lecue_Mendelson:15} derive the
following result, describing the rate of convergence of ERM in linear
aggregation.

\begin{theorem}[\protect\cite{Lecue_Mendelson:15}]
\label{theorem_A}Let $S=\left\{ s_{1},...,s_{D}\right\} \subset L_{2}\left(
P^{X}\right) $ be a dictionary and assume that $m=%
\Span%
\left( S\right) $ satisfies the small-ball condition with constants $\kappa
_{0}$ and $\beta _{0}$ (see Definition \ref{def_small_ball}\ above). Let $%
n\geq \left( 400\right) ^{2}D/\beta _{0}^{2}$ and set $\zeta =Y-s_{m}\left(
X\right) $, where $s_{m}$ is the projection of the regression function $%
s_{\ast }$ onto $m$. Assume further that one of the following two conditions
holds:

\begin{enumerate}
\item $\zeta $ is independent of $X$ and $\mathbb{E}\zeta ^{2}\leq \sigma
^{2}$, or

\item $\left\vert \zeta \right\vert \leq \sigma $ almost surely.
\end{enumerate}

Then the least-squares estimator $\hat{s}_{m}$ on $m$, defined in (\ref%
{def_ERM}), satisfies for every $x>0$, with probability at least $1-\exp
\left( -\beta _{0}^{2}n/4\right) -\left( 1/x\right) $,%
\begin{equation}
\left\Vert \hat{s}_{m}-s_{m}\right\Vert _{2}^{2}\leq \left( \frac{16}{\beta
_{0}\kappa _{0}^{2}}\right) ^{2}\frac{\sigma ^{2}Dx}{n}\text{ .}
\label{bound_small_ball}
\end{equation}
\end{theorem}

Notice that Alternative 1 in Theorem \ref{theorem_A} is equivalent to
assuming that the regression function belongs to $m$ - that is $s_{\ast
}=s_{m}$ - and that the noise is independent from the design - that is $%
\sigma \left( X\right) \equiv \sigma $ is homoscedastic and $\varepsilon $
is independent of $X$ in relation (\ref{reg_rel}).

The main feature of Theorem \ref{theorem_A} is that if the small-ball
condition is achieved with \textit{absolute constants} $\kappa _{0}$ and $%
\beta _{0}$ not depending on the dimension $D$ nor the sample size $n$, then
optimal linear aggregation rates of order $D/n$ are recovered by ERM. If
moreover the regression function belongs to $m$ (Alternative 1), then the
only moment assumption required is that the noise is in $L_{2}$. Otherwise,
Alternative 2 asks for a uniformly bounded noise. Some variants of Theorem %
\ref{theorem_A} are also presented in \cite{Lecue_Mendelson:15}, showing for
instance that optimal rates can be also derived for ERM when the noise as a
fourth moment.

In the analysis of optimal rates in linear aggregation, it is thus worth
understanding when the small ball condition stated in Definition \ref%
{def_small_ball} is achieved with absolute constants.

One such typical situation is for \textit{linear measurements}, that is when
the functions of the dictionary are of the form $f_{i}\left( x\right)
=x^{T}t_{i}$, $t_{i}\in \mathbb{R}^{d}$. Indeed, very weak conditions are
asked on the design $X$ in this case to ensure the small-ball property: for
instance, it suffices to assume that $X$ has independent coordinates that
are absolutely continuous with respect to the Lebesgue measure, with a
density almost surely bounded (see \cite{LecMendel:14a} and \cite{Mendel:15}%
, Section 6, for more details). As shown in \cite{LecMendel:14a} and \cite%
{LecueMendel:16}, this implies that the small-ball property has important
consequences in sparse recovery and analysis of regularized linear
regression.

The constants $\left( \kappa _{0},\beta _{0}\right) $\ of the small-ball
condition influence the rate of convergence exposed in Theorem \ref%
{theorem_A} above through the term $V_{0}:=\beta _{0}^{-2}\kappa _{0}^{-4}$
and therefore, we will provide upper and lower bounds for $V_{0}$ in the
following section for various functional dictionaries.

\subsection{The constants in the small-ball condition for general linear
bases\label{ssection_linear_bases}}

Besides linear measurements discussed in Section \ref{ssection_rate}\ above,
an important class of dictionaries for the linear aggregation problem
consists in expansions along orthonormal bases of $L_{2}\left( P^{X}\right) $%
, which typically correspond to nonparametric estimation.

Our goal in this section is thus to investigate the behavior of the
small-ball condition for some classical orthonormal bases such as piecewise
polynomial functions, including histograms, wavelets or the Fourier basis.

\subsubsection{Some generic limits for the small-ball method\label%
{ssection_limits}}

Let us begin with a general proposition, describing some upper bounds for
the constants $\kappa _{0}$ and $\beta _{0}$ appearing in the small-ball
condition (\ref{def_small_ball}). This will enable us to deduce lower bounds
for the parameter $V_{0}=\beta _{0}^{-2}\kappa _{0}^{-4}$ appearing in the
rate (\ref{bound_small_ball}) of Theorem \ref{theorem_A}\ above and
therefore, we will have some insights on the limits of the small-ball method
for linear aggregation.

One can easily see from its definition that the small-ball condition is more
difficult to ensure, at a heuristic level, when the model at hand contains
some "picky" functions. The following proposition provides some
quantifications of this fact.

\begin{proposition}
\label{prop_upper_bound_constants}Assume that a model $m$ satisfies the
small-ball condition (\ref{def_small_ball}) with constants $\left( \beta
_{0},\kappa _{0}\right) $. Then, it holds%
\begin{eqnarray}
\beta _{0} &\leq &\inf_{f\in m\backslash \left\{ 0\right\} }\mathbb{P}\left(
f\left( X\right) \neq 0\right) \text{ ,}  \label{ineq_beta_0} \\
\kappa _{0} &\leq &\inf_{f\in m\backslash \left\{ 0\right\} }\frac{%
\left\Vert f\right\Vert _{\infty }}{\left\Vert f\right\Vert _{2}}
\label{ineq_kappa_0}
\end{eqnarray}%
and, for any $q>0$,%
\begin{equation}
\beta _{0}\kappa _{0}^{q}\leq \inf_{f\in m\backslash \left\{ 0\right\}
}\left( \frac{\left\Vert f\right\Vert _{q}}{\left\Vert f\right\Vert _{2}}%
\right) ^{q}\text{ .}  \label{ineq_constants_small_ball_q}
\end{equation}%
In particular, we always have $\beta _{0}\kappa _{0}^{2}\leq 1$ and if $m$
contains the constant functions, then $\kappa _{0}\leq 1$.
\end{proposition}

It is interesting to note that Inequalities (\ref{ineq_beta_0}) and (\ref%
{ineq_kappa_0}) are two limiting cases of (\ref{ineq_constants_small_ball_q}%
), respectively when $q\rightarrow 0$ and when $q\rightarrow +\infty $. The
proof of Proposition \ref{prop_assump_small_ball}, which is elementary, is
given in Section \ref{section_proofs_sec_2} below.

It is also worth noticing that the inequality $\beta _{0}\kappa _{0}^{2}\leq
1$ implies that the upper bound of Theorem \ref{theorem_A} - obtained in 
\cite{Lecue_Mendelson:15} - is always greater than $256\sigma ^{2}Dx/n$.

Furthermore, consider a model of histograms on a regular partition $\Pi $ of 
$\mathcal{X=}\left[ 0,1\right] ^{d}$ made of $D$ pieces, $X$ being uniformly
distributed on $\mathcal{X}$. More precisely, for any $I\in \Pi $, set%
\begin{equation*}
s_{I}=\frac{\mathbf{1}_{I}}{\sqrt{P^{X}\left( I\right) }}=\sqrt{D}\mathbf{1}%
_{I}
\end{equation*}%
and take a dictionary $S=\left\{ s_{I};\text{ }I\in \Pi \right\} $,
associated to the model $m=%
\Span%
\left( S\right) $.

Then by Inequality (\ref{ineq_beta_0}), one directly gets $\beta _{0}\leq
D^{-1}$ and as $m$ contains the constants, it holds $V_{0}=\beta
_{0}^{-2}\kappa _{0}^{-4}\geq D^{2}$ and the upper bound (\ref%
{bound_small_ball}) of Theorem \ref{theorem_A} is greater than $256\sigma
^{2}D^{3}x/n$. Hence, the rate of convergence exhibited by the small-ball
method in the case of regular histograms is $D^{3}/n$, which is suboptimal
since it has been proved in 
\cite{ArlotMassart:09,saum:12}
that the excess risk concentrates in this case, under Alternative 2 of
Theorem \ref{theorem_A} above, around a value exactly equal to $\mathbb{E}%
\left[ \zeta ^{2}\right] D/n$.

More generally, when considering the case of a linear model made of
piecewise polynomial functions of degrees bounded by a constant $r$ on a
regular partition, we easily deduce from the previous results on histograms
- that is polynomials of degree zero - that $\beta _{0}\leq rD^{-1}$ for any 
$\kappa _{0}\in \left( 0,1\right) $. We thus have $V_{0}\geq r^{-2}D^{2}$
and the rate of convergence ensured by Theorem \ref{theorem_A} in this case
is again proportional to $D^{3}/n$. It is again suboptimal, since it is also
proved in \cite{saum:12} that for such models of piecewise polynomial
functions, the excess risk concentrates around $\mathbb{E}\left[ \zeta ^{2}%
\right] D/n$, under Alternative 2 of Theorem \ref{theorem_A} above.

Let us now discuss the case of a dictionary made of compactly supported
wavelets.

To fix ideas, let us more precisely state some notations (for more details
about wavelets, see for instance \cite{HKPT:98}). We consider in this case
that $\mathcal{X=}\left[ 0,1\right] $ and $X$ is uniformly distributed on $%
\mathcal{X}$. Set $\phi _{0}$ the father wavelet and $\psi _{0}$ the mother
wavelet. For every integers $j\geq 0$, $1\leq k\leq 2^{j}$, define%
\begin{equation}
\psi _{j,k}:x\mapsto 2^{j/2}\psi _{0}\left( 2^{j}x-k+1\right) \text{ .}
\label{multi_echelle}
\end{equation}

As explained in \cite{CohenDaubVial:93}, there exists several ways to
consider wavelets on the interval. We apply here one of the most classical
construction, that consists in using "periodized" wavelets. To this aim, we
associate to a function $\rho $ on $\mathbb{R}$, the $1$-periodic function%
\begin{equation*}
\rho ^{\text{per}}\left( x\right) =\sum_{p\in \mathbb{Z}}\rho \left(
x+p\right) \text{ .}
\end{equation*}%
Notice that if $\psi $ has a compact support, then the sum at the right-hand
side of the latter inequality is finite for any $x$.

We set for every integer $j\geq 0$, $\Lambda \left( j\right) =\left\{ \left(
j,k\right) \text{ };\text{ }1\leq k\leq 2^{j}\right\} $. Moreover, we set $%
\psi _{-1,1}\left( x\right) =\phi _{0}\left( x\right) ,$ $\Lambda \left(
-1\right) =\left\{ \left( -1,1\right) \right\} $ and for any integer $l\geq
0 $, $\Lambda _{l}=\bigcup\limits_{j=-1}^{l}\Lambda \left( j\right) $. Then
we consider the dictionary $S=\left\{ \psi _{\lambda }^{\text{per}}\text{ };%
\text{ }\lambda \in \Lambda _{l}\right\} $ associated to the model $m=%
\Span%
\left( S\right) $.

Now, it is easily seen from (\ref{multi_echelle}), that for any $\left(
l,k\right) \in \Lambda \left( l\right) $, $\mathbb{P}\left( \left\vert \psi
_{l,k}\left( X\right) \right\vert \not=0\right) \lesssim 2^{-l}\lesssim
D^{-1}$, where $D$ is the linear dimension of $m$. Consequently, as for
histograms and piecewise polynomials on regular partitions, dictionaries
made of compactly supported wavelets are handled through the small-ball
method with a bound proportional to $D^{3}/n$. This rate is again
suboptimal, as shown quite recently by Navarro and Saumard \cite{NavSaum:16}%
, who proved that for such models, the excess risk of the least-squares
estimator concentrates, under Alternative 2 of Theorem \ref{theorem_A}
above, around $\mathbb{E}\left[ \zeta ^{2}\right] D/n$.

It is worth noting that more general multidimensional wavelets could also be
considered at the price of more technicalities.

Wavelets, histograms and piecewise polynomials are models that are \textit{%
formed from} "picky" functions, it is thus quite legitimate that the
small-ball method implies suboptimal rates for these models. What happens
when the dictionary is formed from spatially unlocalized functions such as
the Fourier basis ?

\begin{proposition}
\label{prop_Fourier}Assume that $\mathcal{X=}\left[ -\pi ,\pi \right] $ and
that the design $X$ is uniformly distributed on $\mathcal{X}$. Then Fourier
expansions (i.e. the set of trigonometric polynomials) can not satisfy the
small ball condition (\ref{ineq_small_ball}) with some absolute constants $%
\left( \beta _{0},\kappa _{0}\right) $. More precisely, let us set $\varphi
_{0}\equiv 1$, $\varphi _{2k}\left( x\right) =\sqrt{2}\cos \left( kx\right) $
and $\varphi _{2k+1}\left( x\right) =\sqrt{2}\sin \left( kx\right) $ for $%
k\geq 1$, and take for some $l\in \mathbb{N}$, the model $m_{D}=%
\Span%
\left\{ \varphi _{j};\text{ }j=0,...,2l\right\} $, of linear dimension $%
D=2l+1$. If $m_{D}$ satisfies the small ball condition with constants $%
\left( \beta _{0},\kappa _{0}\right) $, then it holds $\beta _{0}\leq
C\kappa _{0}^{-1/2}D^{-3/4}$ for some absolute constant $C>0$ ($C=3^{1/4}%
\sqrt{2}$ works).
\end{proposition}

\begin{corollary}
\label{cor_fourier}When the design is uniform on $\mathcal{X=}\left[ -\pi
,\pi \right] $, the dictionary is made of the first $D$ elements of the
Fourier basis, the bound given in the right-hand side of Inequality \ref%
{bound_small_ball} in Theorem \ref{theorem_A} above (i.e. Theorem A of \cite%
{Lecue_Mendelson:15}) is bounded from below as follows,%
\begin{equation*}
\left( \frac{16}{\beta _{0}\kappa _{0}^{2}}\right) ^{2}\frac{\sigma ^{2}Dx}{n%
}\geq 256\frac{\sigma ^{2}D^{5/2}x}{n}\text{ .}
\end{equation*}
\end{corollary}

Corollary \ref{cor_fourier} shows that the rate of convergence provided by
the small ball method (Theorem A of \cite{Lecue_Mendelson:15}) is at most $%
D^{5/2}/n$ in the case of the Fourier dictionary. Therefore, we will show in
Section \ref{section_opt_bounds_Fourier} that, under Alternative 2 of
Theorem \ref{theorem_A} above, the excess risk of the least-squares
estimator concentrates around $\mathbb{E}\left[ \zeta ^{2}\right] D/n$, just
as for localized bases such as wavelets, histograms and piecewise
polynomials. Hence, the small-ball method as developed in \cite%
{Lecue_Mendelson:15} gives suboptimal results for the linear aggregation of
the Fourier dictionary.

The proof of Proposition \ref{prop_Fourier} is based on the use of a "picky"
trigonometric polynomial and can be found in Section \ref%
{section_proofs_sec_2}. In Section \ref{ssection_lower_bounds} below, we
will derive a quite general lower bound of the order $D^{-1}$ for $\beta
_{0} $. This bound is in particular valid for the Fourier dictionary, but
does not match with the upper bound decaying like $D^{-3/4}$ of Proposition %
\ref{prop_Fourier}. Therefore, an interesting open question is to determine
what is the exact rate of $\beta _{0}$ with respect to $D$ in the Fourier
case (at a fixed value of $\kappa _{0}$) ? This question remains open.

Finally, it is important to note that Proposition \ref{prop_Fourier} above
is a new result, that may be of some informal interest in the related
context of Fourier measurement matrices for compressed sensing, where a
small-ball condition (or a slightly modified version of it) would yield
optimal recovery rates, as noted by Lecu\'{e} and Mendelson in \cite%
{LecMendel:14a}, Remark 1.5:

\begin{quotation}
\textit{One may wonder if the small-ball condition is satisfied for more
structured matrices, as the argument we use here does not extend immediately
to such cases. And, indeed, for structured ensembles one may encounter a
different situation: a small-ball condition that is not uniform, in the
sense that the constants [...] are direction-dependent.}
\end{quotation}

Concerning instances of "more structured matrices", Lecu\'{e} and Mendelson
add that "one notable example is a random Fourier measurement matrix", which
is designed by randomly selecting rows of a complete discrete Fourier
measurement matrix.

In our setting, also dealing with the Fourier basis but in the "continuous"
setting rather than discrete, we show that indeed, \textit{the small-ball
condition cannot be satisfied for constants }$\left( \kappa _{0},\beta
_{0}\right) $\textit{\ that are absolute}, in the sense that they would be
independent of the dimension. But, we also prove in Section \ref%
{ssection_lower_bounds} below that \textit{the small-ball condition is
achieved}, \textit{for some constants that indeed depend on the dimension}.

To recover better estimates, it seems reasonable then to look at a more
refined property and searching for "direction-dependent" estimates as
proposed in \cite{LecMendel:14a} seems a good option. Indeed, it is clear
that in the directions of functions in the dictionary for instance, that is
for trigonometric functions, the constants are absolute. We follow this lead
in Section \ref{ssection_refined_small_ball}\ below, where we indeed prove
optimal rates of convergence for aggregation on the Fourier dictionary.

\subsubsection{Lower bounds for the small-ball coefficients\label%
{ssection_lower_bounds}}

The following assumption, that states the equivalence between the $L_{\infty
}$ and $L_{2}$ norms for functions in the linear model $m$, is satisfied by
many classical functional bases:

\begin{description}
\item[(\textbf{A1})] Take $S=\left\{ s_{1},...,s_{D}\right\} \subset
L_{2}\left( P^{X}\right) $ a dictionary and consider its linear span $m=%
\Span%
\left( S\right) $. Assume that there exists a positive constant $L_{0}$ such
that, for every $s\in m$,%
\begin{equation}
\left\Vert s\right\Vert _{\infty }\leq L_{0}\sqrt{D}\left\Vert s\right\Vert
_{2}\text{ .}  \label{assumpl_A1}
\end{equation}
\end{description}

\bigskip

\begin{remark}
\label{remark_A1}As soon as we are given a finite dimensional vector space
of functions $m$, then it holds%
\begin{equation*}
R_{m}:=\sup_{s\in m,\text{ }s\neq \left\{ 0\right\} }\frac{\left\Vert
s\right\Vert _{\infty }}{\left\Vert s\right\Vert _{2}}<+\infty \text{ ,}
\end{equation*}%
since the sup-norm and the quadratic norm are equivalent on the finite
dimensional space $m$. In other words, Assumption \textbf{(A1) }is satisfied
as soon as $m$ is of finite dimension, with a parameter $L_{0}$ that may
depend on the dimension $D$. Therefore, the strength of Assumption \textbf{%
(A1) }arises when $L_{0}$ can be chosen independent of the dimension.
\end{remark}

Examples of linear models $m$ satisfying Assumption (\textbf{A1}) with an
absolute constant $L_{0}$ are given for instance in \cite{BarBirMassart:99}
and include many classical nonparametric models for functional estimation,
such as histograms and piecewise polynomials on a regular partition,
compactly supported wavelets and the Fourier basis.

It appears that when a model $m$ satisfies Assumption (\textbf{A1}), the
small-ball condition is verified, but with constants that may depend on the
dimension of the model.

\begin{proposition}
\label{prop_assump_small_ball}If a linear model $m$ is of finite linear
dimension, then it achieves the small ball condition (with parameters $%
\left( \kappa _{0},\beta _{0}\right) $ that may depend on the dimension).
More precisely, for any $\kappa _{0}\in \left( 0,1\right) $, $m$ achieves in
that case the small ball condition with parameter $\beta _{0}$ achieving the
following constraint,%
\begin{equation}
\beta _{0}\geq \frac{1-\kappa _{0}^{2}}{R_{m}^{2}}>0\text{ ,}
\label{ineq_Rm}
\end{equation}%
where $R_{m}=\sup_{s\in m,s\neq 0}\left\Vert s\right\Vert _{\infty
}/\left\Vert s\right\Vert _{2}$ is defined in Remark \ref{remark_A1} above.
Consequently, if $m$ satisfies Assumption \textbf{(A1) }then inequality (\ref%
{ineq_small_ball}) of the small-ball condition given in Definition \ref%
{def_small_ball}\ is verified for any $\kappa _{0}\in \left( 0,1\right) $
with $\beta _{0}=\left( 1-\kappa _{0}^{2}\right) L_{0}^{-2}D^{-1}$.
\end{proposition}

The proof of Proposition \ref{prop_assump_small_ball}, detailed in Section %
\ref{section_proofs_sec_2}, is a direct application of Paley-Zygmund's
inequality (see \cite{delaPenaGine:99}). \cite{Lecue_Mendelson:15} also
noticed that more generally, Paley-Zygmund's inequality could be used to
prove the small-ball property when for some $p>2$, the $L_{p}$ and $L_{2}$
norms are equivalent, or also for \textit{subgaussian classes}, where\textit{%
\ }the Orlicz\textit{\ }$\psi _{2}$ norm is controlled by the $L_{2}$ norm,
see \cite{LecMendel:14b}.

These conditions are weaker than the control of the $L_{\infty }$ norm by
the $L_{2}$ norm, however, as proved in the comments of Proposition \ref%
{prop_Fourier} - see Section \ref{ssection_limits} -, the dependence in $D$
for $\beta _{0}$ given in Proposition \ref{prop_assump_small_ball}\ above is
sharp for localized bases such as histograms, piecewise polynomials and
wavelets. Hence, the control of the $L_{\infty }$ norm by the $L_{2}$ norm
is in some way optimal in these cases, and weaker assumptions could not
imply some improvements on the behavior of the small ball property for these
models.

As for the Fourier basis, the conjunction of Propositions \ref{prop_Fourier}%
\ and \ref{prop_assump_small_ball} gives that for such a basis, for any $%
\kappa _{0}\in \left( 0,1\right) $,%
\begin{equation*}
\frac{1-\kappa _{0}^{2}}{2D}\leq \beta _{0}\leq \frac{3^{1/4}\sqrt{2}}{\sqrt{%
\kappa _{0}}D^{3/4}}\text{ ,}
\end{equation*}%
since, for the lower bound, Assumption \textbf{(A1)}\ is satisfied with $%
L_{0}=\sqrt{2}$ (see for instance \cite{BarBirMassart:99}). As detailed in
Section \ref{ssection_limits} above, it is an open question to find the
right dependence in the dimension for $\beta _{0}$. Moreover, some related
questions have a potential impact on compressed sensing theory as developed
in \cite{LecMendel:14a}.

\section{Optimal excess risks bounds for Fourier expansions\label%
{section_opt_bounds_Fourier}}

We have shown in Section \ref{sec_small_ball} that the small-ball condition
is satisfied for linear models such as histograms, piecewise polynomials,
compactly supported wavelets or the Fourier basis, but with constants that
depend on the dimension of the model in such a way that using this condition
to analyze the rate of convergence of ERM on these models may lead to
suboptimal bounds.

Our aim in this section is to show that optimal rates of linear aggregation
can indeed be attained by ERM in the Fourier case, that is when the model $m$
is spanned by the $D$ first elements of the Fourier basis. We consider two
different settings.

In the bounded setting, exposed in Section \ref{ssection_excess_risk_concen}%
, we prove sharp upper and lower bounds for the excess risk that more
precisely ensure its concentration around a single deterministic point.

In the general setting treated in Section \ref{ssection_refined_small_ball},
we refine the small-ball arguments developed in \cite{Lecue_Mendelson:15} by
focusing on certain directions where the small-ball is uniform and we also
obtain optimal rates of linear aggregation when the noise is only assumed to
have a second moment.

\subsection{Excess risk's concentration\label{ssection_excess_risk_concen}}

We focus in this section on the bounded setting. Let us precisely detail our
assumptions. Assume that the design $X$ is uniformly distributed on $%
\mathcal{X=}\left[ 0,2\pi \right] $ and that the regression function $%
s_{\ast }$ satisfies $s_{\ast }\left( 0\right) =s_{\ast }\left( 2\pi \right) 
$. Then the Fourier basis is orthonormal in $L_{2}(P^{X})$ and we consider a
model $m$ of dimension $D$ (assumed to be odd) corresponding to the linear
vector space spanned by the first $D$ elements of the Fourier basis. More
precisely, if we set $\varphi _{1}\equiv 1$, $\varphi _{2k}\left( x\right) =%
\sqrt{2}\cos \left( kx\right) $ and $\varphi _{2k+1}\left( x\right) =\sqrt{2}%
\sin \left( kx\right) $ for $k\geq 1$, then $\left( \varphi _{j}\right)
_{j=1}^{D}$ is an orthonormal basis of $\left( m,\left\Vert \cdot
\right\Vert _{2}\right) $, for an integer $l$ satisfying $2l+1=D$. Assume
also:

\begin{itemize}
\item (\textbf{H1}) The data and the linear projection of the target onto $m$
are bounded by a positive finite constant $A$: 
\begin{equation}
\left\vert Y\right\vert \leq A\text{ }a.s.  \label{bounded_data}
\end{equation}%
and%
\begin{equation}
\left\Vert s_{m}\right\Vert _{\infty }\leq A\text{ }.
\label{bounded_projection}
\end{equation}
\end{itemize}

\noindent Hence, from (\textbf{H1}) we deduce that 
\begin{equation}
\left\Vert s_{\ast }\right\Vert _{\infty }=\left\Vert \mathbb{E}\left[
Y\left\vert X=%
\cdot%
\right. \right] \right\Vert _{\infty }\leq A  \label{bounded_target}
\end{equation}%
and that there exists a constant $\sigma _{\max }>0$ such that 
\begin{equation}
\sigma ^{2}\left( X_{i}\right) \leq \sigma _{\max }^{2}\leq A^{2}\text{ \ }%
a.s.  \label{majo_sigma_max}
\end{equation}

\begin{itemize}
\item (\textbf{H2}) The heteroscedastic noise level $\sigma $ is not reduced
to zero: 
\begin{equation*}
\left\Vert \sigma \right\Vert _{2}=\sqrt{\mathbb{E}\left[ \sigma ^{2}\left(
X\right) \right] }>0\text{ .}
\end{equation*}
\end{itemize}

We are now in position to state our result.

\begin{theorem}
\label{theorem_fourier}Let $A_{+},A_{-},\alpha >0$ and let $m$ be a linear
vector space spanned by a dictionary made of the first $D$ elements of the
Fourier basis\textbf{.} Assume (\textbf{H1-2}) and take $\varphi =\left(
\varphi _{k}\right) _{k=1}^{D}$ the Fourier basis of $m$. If it holds 
\begin{equation}
A_{-}\left( \ln n\right) ^{2}\leq D\leq A_{+}\frac{n^{1/2}}{\ln n}\text{ },
\label{hypo_dim_reg}
\end{equation}%
then there exists a constant $A_{0}>0$, only depending on $\alpha
,A_{-},A_{+}$ and on the constants $A,\left\Vert \sigma \right\Vert _{2}$
defined in assumptions (\textbf{H1-2)}, such that by setting%
\begin{equation}
\varepsilon _{n}=A_{0}\max \left\{ \sqrt{\frac{\ln n}{D}},\text{ }\frac{D}{%
\sqrt{n}}\right\} \text{ },  \label{def_A_0}
\end{equation}%
we have for all $n\geq n_{0}\left( \alpha \right) $,%
\begin{equation}
\mathbb{P}\left[ \left( 1-\varepsilon _{n}\right) \frac{D}{n}\mathcal{C}%
_{m}^{2}\leq \left\Vert \hat{s}_{m}-s_{m}\right\Vert _{2}^{2}\leq \left(
1+\varepsilon _{n}\right) \frac{D}{n}\mathcal{C}_{m}^{2}\right] \geq
1-3n^{-\alpha }\text{ },  \label{upper_true_risk}
\end{equation}%
where $\hat{s}_{m}$ is the least-squares estimator on $m$, defined in (\ref%
{def_ERM}), and 
\begin{equation}
\mathcal{C}_{m}^{2}=\mathbb{E}\left[ \sigma ^{2}\left( X\right) \right]
+\left\Vert s_{\ast }-s_{m}\right\Vert _{2}^{2}\text{ .}
\label{def_Cm_Fourier}
\end{equation}
\end{theorem}

The rate of convergence of ERM for linear aggregation with a Fourier
dictionary exhibited by Theorem \ref{theorem_fourier} is thus of the order $%
D/n$, which is the optimal rate of linear aggregation. In particular, this
outperforms the bounds obtained in Theorem 2.2\ of \cite{AudCat:11} under
same assumption as Assumption (\textbf{A1}), that is satisfied in the
Fourier case, but also under more general moment assumptions on the noise.
Indeed, as noticed in \cite{Lecue_Mendelson:15}, the bounds obtained by \cite%
{AudCat:11} are in this case of the order $D^{3}/n$, for models of dimension
lower than $n^{1/4}$. In Theorem \ref{theorem_fourier}, our condition on the
permitted dimension is less restrictive, since models with dimension close
to $n^{1/2}$ are allowed.

Concerning the assumptions, uniform boundedness of the projection of the
target onto the model, as described in (\ref{bounded_projection}), is not so
restrictive and is guaranteed as soon as the regression function belongs to
a broad class of functions named the Wiener algebra, that is whenever the
Fourier coefficients of the regression function are summable (in other words
when the Fourier series of the regression function is absolutely
convergent). For instance, functions that are H\"{o}lder continuous with
index greater than 1/2 belong to the Wiener algebra, \cite%
{katznelson2004introduction}. For more on the Wiener algebra, see Section %
\ref{ssection_refined_small_ball}\ below.

Furthermore, Theorem \ref{theorem_fourier} gives an information that is far
more precise than the rate of convergence of the least-squares estimator.
Indeed, Inequality (\ref{upper_true_risk}) of Theorem \ref{theorem_fourier}
actually proves the \textit{concentration} \textit{of the excess risk} of
the least-squares estimator around one precise value, which is $D\mathcal{C}%
_{m}^{2}/n$.

There are only very few and recent such concentration results for the excess
risk of a M-estimator in the literature and this question constitutes an
exiting new line of research in learning theory. Considering the same
regression framework as ours, \cite{saum:12} has shown concentration bounds
for the excess risk of the least-squares estimator on models of piecewise
polynomial functions. Furthermore, these results have been recently extended
in \cite{NavSaum:16}\ to strongly localized bases, a class of dictionaries
containing in particular compactly supported wavelets.

In a slightly different context of least-squares estimation under convex
constraint, \cite{chatterjee2014} also proved the concentration in $L_{2}$
norm, with fixed design and Gaussian noise. Under the latter assumptions, 
\cite{MurovandeGeer:15} have shown the excess risk's concentration for the
penalized least-squares estimator. Finally, \cite{vandeGeerWain:16} recently
proved some concentration results for some regularized M-estimators. They
also give an application of their results to a linearized regression context
with random design and independent Gaussian noise.

The proof of Theorem \ref{theorem_fourier} is developed in Section \ref%
{section_opt_bounds_Fourier}. We make a recurrent use along our proofs of
classical Talagrand's type concentration inequalities for suprema of the
empirical process with bounded arguments. We also make use of other tools
from empirical process theory, such as a control of variance of the
empirical process with bounded arguments - see the proof of Theorem \ref%
{cor_alt_hoff_2} in Section \ref{section_probabilistic_tools}.

\subsection{A refined small-ball argument\label{ssection_refined_small_ball}}

As proved in Section \ref{sec_small_ball}\ above, a direct application of
results of \cite{Lecue_Mendelson:15} can not lead to the optimal rate of
convergence for linear aggregation \textit{via }empirical risk minimization
on the Fourier dictionary.

To recover better estimates, it seems reasonable then to look at a more
refined property and searching for "direction-dependent" estimates as
proposed in \cite{LecMendel:14a} - see the quotation in Section \ref%
{ssection_limits} above - seems a good option. Indeed, it is clear that in
the directions of functions in the dictionary for instance, that is for
trigonometric functions, the constants are absolute. We follow here this
lead and this enables us to prove optimal rates of convergence for linear
aggregation on the Fourier dictionary.

As explained in the comments following Theorem \ref{theorem_fourier} above,
the assumptions needed for Theorem \ref{theorem_fourier} and especially
Assumption (\ref{bounded_projection}) of uniform boundedness of the
projection of the regression function, are ensured if the target belongs to
the Wiener algebra, that is if Fourier coefficients are summable. In this
case of course, the projection of the target on a Fourier dictionary (with
any cardinality) is again in the Wiener algebra. We now denote $A\left( 
\mathbb{T}\right) $ the Wiener algebra. It holds, by definition,%
\begin{equation*}
A\left( \mathbb{T}\right) =\left\{ f=\sum_{k\geq 1}\beta _{k}\varphi _{k}%
\text{ };\text{ }\sum_{k\geq 1}\left\vert \beta _{k}\right\vert <+\infty
\right\} \text{ .}
\end{equation*}

We look here at some subsets of the Wiener algebra.

\begin{definition}
\label{def_Gamma_set}Let us take $\nu >0$ and denote, for a function $f$ $%
2\pi $-periodic, $\beta _{k}(f)=\left\langle f,\varphi _{k}\right\rangle $.
We define the set%
\begin{equation*}
\Lambda _{\nu }\left( L_{1},L_{2}\right) =\left\{ f\in L_{\infty }\left( 
\mathbb{T}\right) \text{ };\text{ }\sum_{k\geq 1}k^{\nu }\left\vert \beta
_{k}\left( f\right) \right\vert \leq L_{1}\text{ }\And \text{\ }\left\Vert
f\right\Vert _{\infty }\geq L_{2}\right\} \text{ .}
\end{equation*}
\end{definition}

In the perspective of the small-ball approach, the interest of the set $%
\Lambda _{\nu }\left( L_{1},L_{2}\right) $ lies in the following
proposition, ensuring that the small-ball condition (\ref{def_small_ball})
is fulfilled uniformly on $\Lambda \left( L_{1},L_{2}\right) $ whenever $\nu
>1/2$ and $L_{2}>0$, for some constants $\left( \kappa _{0},\beta
_{0}\right) $ that only depend on $\nu ,L_{1}$ and $L_{2}$.

\begin{proposition}
\label{prop_small_ball_Gamma}Fix $\nu >1/2$ and $L_{1},L_{2}>0.$ Take some
function $f\in \Lambda _{\nu }\left( L_{1},L_{2}\right) $. Then for any $%
\kappa _{0}\in \left( 0,1\right) $, it holds%
\begin{equation*}
\mathbb{P}\left( \left\vert f\left( X\right) \right\vert \geq \kappa
_{0}\left\Vert f\right\Vert _{2}\right) \geq \frac{\left( 1-\kappa
_{0}^{2}\right) }{4C_{\nu }^{2}}\frac{L_{2}^{2}}{L_{1}^{4}}>0\text{ ,}
\end{equation*}%
with $C_{\nu }=\sum_{k\geq 1}k^{-2\nu }<+\infty $. In other words, the
small-ball condition (\ref{def_small_ball}) is satisfied uniformly over $%
\Lambda _{\nu }\left( L_{1},L_{2}\right) $ with constants $\left( \kappa
_{0},\beta _{0}\right) ,$ for $\kappa _{0}\in \left( 0,1\right) $ and $\beta
_{0}=C_{\nu }^{-2}L_{2}^{2}L_{1}^{-4}\left( 1-\kappa _{0}^{2}\right) /4$.
\end{proposition}

It is clear from Definition \ref{def_Gamma_set} that for any $\nu >0$, $%
\Lambda _{\nu }\left( L_{1},L_{2}\right) \subset A\left( \mathbb{T}\right) $%
. Furthermore, any function of sup-norm greater than the constant $L_{2}$
and belonging to a (periodic) Sobolev space $W_{\gamma }$ of parameter $%
\gamma $ belongs to $\Lambda _{\nu }\left( L_{1},L_{2}\right) $, for some
constant $L_{1}$ and $\nu <\gamma -1/2$.

Recall that periodic Sobolev spaces $W_{\gamma }:=\dbigcup_{L>0}W\left(
\gamma ,L\right) $ are defined as follows (see for instance \cite{Tsy:04b},
Section 1.10), for any $\gamma \in \mathbb{N}_{\ast }$,%
\begin{equation*}
\begin{array}{l}
W\left( \gamma ,L\right) :=\left\{ f\in L_{2}\left( \mathbb{T}\right) \text{ 
};\text{ }f^{\left( \gamma -1\right) }\text{ is absolutely continuous,}%
\right. \\ 
\text{ \ \ \ \ \ \ \ \ \ \ \ \ \ \ }\left. \text{ }\frac{1}{2\pi }%
\int_{0}^{2\pi }\left( f^{\left( \gamma \right) }\left( x\right) \right)
^{2}dx\leq L\text{ }\&\text{ }f^{\left( j\right) }\left( 0\right) =f^{\left(
j\right) }\left( 1\right) ,\text{ }j=0,1,...,\gamma -1\right\} \text{ .}%
\end{array}%
\end{equation*}%
In addition, the regularity of periodic functions in Sobolev spaces can be
directly read on the order of magnitude of their Fourier coefficients. More
precisely, for any $\gamma \in \mathbb{N}_{\ast }$, $W_{\gamma
}=\dbigcup_{L>0}\tilde{W}\left( \gamma ,Q\right) $, where%
\begin{equation*}
\tilde{W}\left( \gamma ,Q\right) :=\left\{ f\in L_{2}\left( \mathbb{T}%
\right) \text{ };\text{ }f=\sum_{k\geq 1}\beta _{k}\varphi _{k}\text{ }\&%
\text{ }\sum_{k\geq 1}k^{2\gamma }\beta _{k}^{2}\leq Q\right\} \text{ .}
\end{equation*}%
This second characterization of Sobolev spaces $W_{\gamma }$ allow to extend
their definition to any $\gamma >0$ and not only to integer valued $\gamma $%
. Thus, this is the definition we use in the following proposition.

\begin{proposition}
\label{prop_sob_Gam}With the previous notations, it holds for any $\nu >0$,%
\begin{equation*}
\left[ \dbigcup_{\left\{ \gamma \text{ }:\text{ }1/2+\nu <\gamma \right\} }%
\tilde{W}\left( \gamma ,Q\right) \dbigcap \left\{ f\in L_{\infty }\left( 
\mathbb{T}\right) \text{ };\text{ }\left\Vert f\right\Vert _{\infty }\geq
L_{2}\right\} \right] \subset \Lambda _{\nu }\left( L_{1},L_{2}\right) \text{
,}
\end{equation*}%
whenever $Q\leq L_{1}^{2}\left( \sum_{k\geq 1}k^{2\left( \nu -\gamma \right)
}\right) ^{-1}<+\infty $.
\end{proposition}

Proposition \ref{prop_sob_Gam}\ is appealing since the Fourier dictionary is
known to achieve minimax rates of convergences for the estimation of a
regression function, whenever it lies in a Sobolev space $W_{\gamma }$ of
parameter $\gamma >1$ (\cite{Tsy:04b}). Indeed, by Proposition \ref%
{prop_small_ball_Gamma}, we are interested by the sets $\Lambda _{\nu
}\left( L_{1},L_{2}\right) $ for $\nu >1/2$ and Proposition \ref%
{prop_sob_Gam} implies that such sets contain function of Sobolev regularity 
$\gamma >\nu +1/2>1$. This latter fact thus legitimate the focus on the sets 
$\Lambda _{\nu }\left( L_{1},L_{2}\right) ,$ $\nu >1/2$, to deal with the
performance of linear aggregation from the Fourier dictionary.

Let us turn now to the main result of this section.

\begin{theorem}
\label{theorem_small_ball_modif_Fourier}Fix $\nu >1/2$, $L_{1},L_{2}>0$ and
assume that $s_{\ast }\in \Lambda _{\nu }\left( L_{1},L_{2}\right) $. Let $%
S=\left\{ \varphi _{1},...,\varphi _{D}\right\} $ be a dictionary made of
the $D$ first elements of the Fourier basis. Set $\zeta =Y-s_{m}\left(
X\right) $, where $s_{m}$ is the projection of the regression function $%
s_{\ast }$ onto $m$. Assume that $\zeta $ is independent of $X$ and $\mathbb{%
E}\zeta ^{2}\leq \sigma ^{2}$. Then there exists three constants $%
L_{v},L_{L_{1},L_{2},\sigma ,\nu },$ $C_{\nu ,L_{1},L_{2}}>0$ and an integer 
$n_{0}\left( \nu ,L_{1},L_{2}\right) $ such that, if%
\begin{equation}
0<\left( 2\sqrt{2}L_{1}L_{2}^{-1}\right) ^{1/\nu }\leq D\leq L_{\nu }\left(
n/\ln n\right) ^{\frac{1}{2\left( \nu +1\right) }}
\label{condition_Dimension}
\end{equation}%
and $x\in \left( 0,L_{L_{1},L_{2},\sigma ,\nu }n/D^{2\left( \nu +1\right)
}\right) $, the least-squares estimator $\hat{s}_{m}$ on $m$, defined in (%
\ref{def_ERM}), satisfies for any $n\geq n_{0}\left( \nu ,L_{1},L_{2}\right) 
$, on an event of probability at least $1-\exp \left( -\beta
_{0}^{2}n/4\right) -n^{-2}-\left( 2/x\right) $,%
\begin{equation}
\left\Vert \hat{s}_{m}-s_{m}\right\Vert _{2}^{2}\leq C_{\nu ,L_{1},L_{2}}%
\frac{\sigma ^{2}Dx}{n}\text{ .}  \label{borne_gene_fourier}
\end{equation}
\end{theorem}

The bound (\ref{borne_gene_fourier}) obtained in Theorem \ref%
{theorem_small_ball_modif_Fourier} is optimal in the sense that it achieves
the optimal rate $D/n$ of linear aggregation. Moreover, the only moment
needed on the noise term $\zeta $ is a second moment, which is a minimal
assumption. The proof, exposed in Section \ref%
{section_proof_refined_small_ball}, is based on a localization of the
least-squares estimator on directions of uniform small-ball property.

Compared to Theorem \ref{theorem_fourier}, where we also derived optimal
rates of aggregation, but in the bounded setting, we have a stronger
assumption on the regularity of the target. Indeed, in Theorem \ref%
{theorem_small_ball_modif_Fourier} $s_{\ast }$ is assumed to belong to some $%
\Lambda _{\nu }\left( L_{1},L_{2}\right) ,$ $\nu >1/2$, whereas in Theorem %
\ref{theorem_fourier}, we only assume that the projection $s_{m}$ of the
target onto the model $m$ is uniformly bounded by a constant independent of
the dimension, which is achieved as soon as $s_{\ast }$ belongs to the
Wiener algebra $A\left( \mathbb{T}\right) $. It appears to be the price to
pay to deal with a general noise term, but as explained earlier in this
section, the sets $\Lambda _{\nu }\left( L_{1},L_{2}\right) ,$ $\nu >1/2$,
are natural when dealing with the performance of the Fourier dictionary.

Finally, the range of considered dimensions in (\ref{condition_Dimension})
is fairy reasonable, the upper bound being polynomial in $n$. In addition,
the lower bound, of the order of a constant, is very mild and ensures that
the projection of the regression function onto the model does not vanish in
sup-norm.

\section{Proofs related to Section \protect\ref{sec_small_ball}\label%
{section_proofs_sec_2}}

\begin{proof}[Proof of Proposition \protect\ref{prop_upper_bound_constants}]

Take $f\in m\backslash \left\{ 0\right\} $. Then, it holds%
\begin{equation*}
\beta _{0}\leq \mathbb{P}\left( \left\vert f\left( X\right) \right\vert \geq
\kappa _{0}\left\Vert f\right\Vert _{2}\right) \leq \mathbb{P}\left( f\left(
X\right) \neq 0\right)
\end{equation*}%
which readily gives (\ref{ineq_beta_0}). Furthermore, as $\mathbb{P}\left(
\left\vert f\left( X\right) \right\vert \geq \kappa \left\Vert f\right\Vert
_{2}\right) =0$ for $\kappa >\left\Vert f\right\Vert _{\infty }/\left\Vert
f\right\Vert _{2}$, it holds $\kappa _{0}\leq \left\Vert f\right\Vert
_{\infty }/\left\Vert f\right\Vert _{2}$, which implies (\ref{ineq_kappa_0})
by minimizing the latter bound over all $f\in m\backslash \left\{ 0\right\} $%
. Now,%
\begin{eqnarray*}
\beta _{0} &\leq &\mathbb{P}\left( \left\vert f\left( X\right) \right\vert
\geq \kappa _{0}\left\Vert f\right\Vert _{2}\right) =\mathbb{P}\left( \frac{%
\left\vert f\left( X\right) \right\vert }{\kappa _{0}\left\Vert f\right\Vert
_{2}}\geq 1\right) \\
&\leq &\int_{\mathcal{X}}\left( \frac{\left\vert f\left( x\right)
\right\vert }{\kappa _{0}\left\Vert f\right\Vert _{2}}\right)
^{q}dP^{X}\left( x\right) =\left( \frac{\left\Vert f\right\Vert _{q}}{\kappa
_{0}\left\Vert f\right\Vert _{2}}\right) ^{q}\text{ ,}
\end{eqnarray*}%
and by taking the infimum over $f\in m\backslash \left\{ 0\right\} $, this
thus proves (\ref{ineq_constants_small_ball_q}) and imply $\beta _{0}\kappa
_{0}^{2}\leq 1$ for $q=2$. Finally, when $m$ contains the function
identically equal to one, then $\inf_{f\in m\backslash \left\{ 0\right\}
}\left\Vert f\right\Vert _{\infty }\left\Vert f\right\Vert _{2}^{-1}=1$,
which implies $\kappa _{0}\leq 1$.
\end{proof}

\bigskip

\begin{proof}[Proof of Proposition \protect\ref{prop_Fourier}]
Recall that $D=2l+1$ is the linear dimension of $m_{D}$ and take $\left(
\beta _{0},\kappa _{0}\right) $ satisfying the small ball condition on $%
m_{D} $. Define the $l$th Fej\'{e}r kernel $F_{l}$ as follows,%
\begin{equation*}
F_{l}\left( t\right) =\left\{ 
\begin{array}{ll}
\frac{\sin ^{2}\left( \left( l+1\right) t/2\right) }{\left( l+1\right) \sin
^{2}\left( t/2\right) } & ,\text{ }t\in \left. \left[ -\pi ,\pi \right]
\right\backslash \left\{ 0\right\} \\ 
l+1 & ,\text{ }t=0%
\end{array}%
\right. \text{ .}
\end{equation*}%
Properties of $F_{l}$ are well-known, see for instance \cite%
{BachmNariciBecken:00}, Section 4.15. In particular, $F_{l}\in m_{D}$, $%
F_{l}\geq 0$, $\left\Vert F_{l}\right\Vert _{\infty }\leq l+1=\left(
D+1\right) /2$. Furthermore, $\int_{-\pi }^{\pi }F_{l}\left( t\right)
dt=2\pi $ which by positivity of $F_{l}$ gives $\left\Vert F_{l}\right\Vert
_{1}=1$. We also have, for all $t\in \left[ -\pi ,\pi \right] $,%
\begin{equation*}
F_{l}\left( t\right) =\sum_{k=-l+1}^{l-1}\left( 1-\frac{\left\vert
k\right\vert }{l}\right) e^{ikt}\text{ .}
\end{equation*}%
Using this formula, one easily computes the quadratic norm of the Fej\'{e}r
kernel,%
\begin{equation*}
\left\Vert F_{l}\right\Vert _{2}^{2}=1+2\sum_{k=1}^{l-1}\left( 1-\frac{k}{l}%
\right) ^{2}=1+\frac{2}{l^{2}}\sum_{j=1}^{l-1}j^{2}\sim \frac{2l}{3}\text{ .}
\end{equation*}%
Now, since for any $\varepsilon \in \left( 0,\pi \right] $,%
\begin{equation*}
\sup_{\varepsilon \leq \left\vert t\right\vert \leq \pi }F_{l}\left(
t\right) \leq \frac{1}{l+1}\frac{1}{\sin ^{2}\left( \varepsilon /2\right) }%
\leq \frac{1}{l+1}\left( \frac{\pi }{\varepsilon }\right) ^{2}\text{ ,}
\end{equation*}%
it holds 
\begin{equation*}
\mathbb{P}\left( \left\vert F_{l}\left( X\right) \right\vert \geq \kappa
_{0}\left\Vert F_{l}\right\Vert _{2}\right) \leq \left( \kappa
_{0}\left\Vert F_{l}\right\Vert _{2}\left( l+1\right) \right) ^{-1/2}\text{ .%
}
\end{equation*}%
Consequently, $\beta _{0}\leq \left( \kappa _{0}\left\Vert F_{l}\right\Vert
_{2}\left( l+1\right) \right) ^{-1/2}\sim 3^{1/4}\sqrt{2}\kappa
_{0}^{-1/2}D^{-3/4}$, which gives the result.
\end{proof}

\bigskip

\bigskip

\begin{proof}[Proof of Corollary \protect\ref{cor_fourier}]
From Proposition \ref{prop_Fourier}, it holds%
\begin{equation*}
\beta _{0}\kappa _{0}^{1/2}\leq CD^{-3/4}\text{ .}
\end{equation*}%
Furthermore, as the model contains the constants, we have $\kappa _{0}\leq 1$
and by combining the two inequalities, $\beta _{0}\kappa _{0}^{2}\leq
CD^{-3/4}$, which gives the result.
\end{proof}

\bigskip

\bigskip

\begin{proof}[Proof of Proposition \protect\ref{prop_assump_small_ball}]
Take $s\in m\backslash \left\{ 0\right\} $ and $\kappa _{0}\in \left(
0,1\right) $. Set $\Omega _{\kappa _{0}}\left( s\right) =\left\{ \left\vert
s\left( X\right) \right\vert \geq \kappa _{0}\left\Vert s\right\Vert
_{2}\right\} $. By Paley-Zygmund's inequality (Corollary 3.3.2 in \cite%
{delaPenaGine:99}), it holds%
\begin{equation*}
\mathbb{P}\left( \Omega _{\kappa _{0}}\left( s\right) \right) \mathbb{\geq }%
\left( 1-\kappa _{0}^{2}\right) \frac{\left\Vert s\right\Vert _{2}^{2}}{%
\left\Vert s\right\Vert _{\infty }^{2}}\geq \frac{1-\kappa _{0}^{2}}{%
R_{m}^{2}}\text{ ,}
\end{equation*}%
which gives (\ref{ineq_Rm}). The rest of Proposition \ref%
{prop_assump_small_ball} follows from the latter bound via a simple
application of assumption \textbf{(A1) }to bound from above the term $R_{m}$.
\end{proof}

\section{Proofs related to Section \protect\ref{section_opt_bounds_Fourier} 
\label{section_proof_fourier}}

\subsection{Proof of Theorem \protect\ref{theorem_fourier}\label%
{ssection_proof_th_fourier}}

Aiming at clarifying the proofs, we generalize a little bit the Fourier
framework by invoking the following assumption, that is satisfied for
Fourier expansions. From now on, $m\subset L_{2}\left( P^{X}\right) $ is
considered to be a linear model of dimension $D$, not necessarily built from
the Fourier basis.

\begin{itemize}
\item (\textbf{H3}) Uniformly bounded basis : there exists an orthonormal
basis $\varphi =\left( \varphi _{k}\right) _{k=1}^{D}$ in $\left(
m,\left\Vert 
\cdot%
\right\Vert _{2}\right) $ that satisfies, for a positive constant $u_{m}$, 
\begin{equation*}
\left\Vert \varphi _{k}\right\Vert _{\infty }\leq u_{m}\text{ .}
\end{equation*}
\end{itemize}

Notice that in the Fourier case, (\textbf{H3}) is valid by taking $u_{m}\leq 
\sqrt{2}$.

\begin{remark}
\label{remark_H4}By Cauchy-Schwarz inequality, we also see that when (%
\textbf{H3}) is valid, it holds
\end{remark}

\begin{equation}
\sup_{s\in m,\text{ }\left\Vert s\right\Vert _{2}\leq 1}\left\Vert
s\right\Vert _{\infty }\leq u_{m}\sqrt{D}\text{.}
\label{control_boule_unite}
\end{equation}

Let us denote $\psi _{m}\left( x,y\right) =y-s_{m}\left( x\right) $. Then,
if $\left( \varphi _{k}\right) _{k=1}^{D}$ is formed by the first $D$
elements of the Fourier basis, the quantity $\mathcal{C}_{m}$ defined in (%
\ref{def_Cm_Fourier}) satisfies%
\begin{equation}
\mathcal{C}_{m}^{2}=\frac{1}{D}\sum_{k=0}^{D-1}%
\var%
\left( \psi _{m}%
\cdot%
\varphi _{k}\right) \text{ .}  \label{def_Cm_gene}
\end{equation}%
We will thus prove a slightly more general version than Theorem \ref%
{theorem_fourier}, assuming that (\textbf{H3}) holds and proving Inequality (%
\ref{upper_true_risk}) with the term $\mathcal{C}_{m}$ given by (\ref%
{def_Cm_gene}).

We are now in position to prove Theorem \ref{theorem_fourier}.\bigskip

\begin{proof}[Proof of Theorem \protect\ref{theorem_fourier}]
Take $s=\sum_{k=1}^{D}\beta _{k}\varphi _{k}\in m$. The empirical risk on $s$
writes%
\begin{align*}
& P_{n}\left( \gamma \left( s\right) \right) =P_{n}\left[ \left( y-\left(
\sum_{k=1}^{D}\beta _{k}\varphi _{k}\left( x\right) \right) \right) ^{2}%
\right] \\
& =P_{n}y^{2}-2\sum_{k=1}^{D}\beta _{k}P_{n}\left( y\varphi _{k}\left(
x\right) \right) +\sum_{k,l=1}^{D}\beta _{k}\beta _{l}P_{n}\left( \varphi
_{k}\varphi _{l}\right) \text{ }.
\end{align*}%
By taking the derivative with respect to $\beta _{l}$ in the last quantity,
we get 
\begin{align}
& \frac{1}{2}\frac{\partial }{\partial \beta _{l}}P_{n}\left[ \left(
y-\left( \sum_{k=1}^{D}\beta _{k}\varphi _{k}\left( x\right) \right) \right)
^{2}\right]  \notag \\
& =-P_{n}\left( y\varphi _{l}\left( x\right) \right) +\sum_{k=1}^{D}\beta
_{k}P_{n}\left( \varphi _{k}\varphi _{l}\right) \text{ .}  \label{derivative}
\end{align}%
Hence, we see that if $\hat{\beta}_{m}=\left( \hat{\beta}_{k}\right)
_{k=1}^{D}\in \mathbb{R}^{D}$ is a critical point of the empirical risk
(seen as a function on $\mathbb{R}^{D}$), then it satisfies the following
random linear system,%
\begin{equation}
\left( I_{D}+A_{n,D}\right) \hat{\beta}_{m}=E_{y,n}  \label{system_Fourier}
\end{equation}%
where $E_{y,n}=\left( P_{n}\left( y\varphi _{k}\left( x\right) \right)
\right) _{k=1}^{D}\in \mathbb{R}^{D},$ $I_{D}$ is the identity matrix of
dimension $D$ and $A_{n,D}=\left( \left( P_{n}-P\right) \left( \varphi
_{k}\varphi _{l}\right) \right) _{k,l=1,..,D}$ is a $D\times D$ matrix. Now,
by Inequality (\ref{ineq_A_n_D}) in Lemma \ref{lemma_control_A_n_d} below, a
positive integer $n_{0}\left( u_{m},\alpha \right) $ can be found such that
for all $n\geq n_{0}$, we have on an event $\Omega _{n}$ of probability at
least $1-n^{-\alpha }$, 
\begin{equation}
\left\vert \left\Vert A_{n,D}\right\Vert \right\vert \leq
L_{A_{-},u_{m},\alpha }\frac{D}{\sqrt{n}}\leq \frac{1}{2}\text{ },
\label{control_A_n_D}
\end{equation}%
where for a $D\times D$ matrix $A$, the operator norm $\left\vert \left\Vert 
\cdot%
\right\Vert \right\vert $ associated to the quadratic norm $\left\vert \cdot
\right\vert _{2}$ on vectors is%
\begin{equation*}
\left\vert \left\Vert A\right\Vert \right\vert =\sup_{x\neq 0}\frac{%
\left\vert Ax\right\vert _{2}}{\left\vert x\right\vert _{2}}\text{ }.
\end{equation*}%
We restrict now on analysis on the event $\Omega _{n}$. Then we deduce from (%
\ref{control_A_n_D}) that $\left( I_{D}+A_{n,D}\right) $ is a non-singular $%
D\times D$ matrix and, as a consequence, that the linear system (\ref%
{system_Fourier}) admits a unique solution $\hat{\beta}_{m}$ for any $n\geq
n_{0}\left( u_{m},\alpha \right) $. Moreover, since $P_{n}\left( y-\left(
\sum_{k=1}^{D}\beta _{k}\varphi _{k}\left( x\right) \right) \right) ^{2}$ is
a nonnegative quadratic functional with respect to $\left( \beta _{k}\right)
_{k=1}^{D}\in \mathbb{R}^{D}$ we deduce that for any $n\geq n_{0}\left(
u_{m},\alpha \right) $, $\hat{\beta}_{m}$\ achieves on $\Omega _{n}$ the
unique minimum of $P_{n}\left( y-\left( \sum_{k=1}^{D}\beta _{k}\varphi
_{k}\left( x\right) \right) \right) ^{2}$ on $\mathbb{R}^{D}$, thus $\hat{s}%
_{m}=\sum_{k=1}^{D}\hat{\beta}_{k}\varphi _{k}$.

Now, if we denote $\beta _{m}=\left( \beta _{\ast ,k}\right) _{k=1}^{D}$ the
vector such that $s_{m}=\sum_{k=1}^{D}\beta _{\ast ,k}\varphi _{k}$, then
from (\ref{system_Fourier}) we obtain%
\begin{equation*}
\left( I_{D}+A_{n,D}\right) \left( \hat{\beta}_{m}-\beta _{m}\right) =F_{y,n}%
\text{ ,}
\end{equation*}%
where $F_{y,n}:=E_{y,n}-\left( I_{D}+A_{n,D}\right) \beta _{m}\in \mathbb{R}%
^{D}$. Furthermore, straightforward computations give, 
\begin{equation}
F_{y,n}=\left( \left( P_{n}-P\right) \left( \psi _{1,m}\varphi _{k}\right)
\right) _{k=1}^{D}\text{ ,}  \label{def_F_y_n}
\end{equation}%
where $\psi _{1,m}\left( x,y\right) =y-s_{m}\left( x\right) ,$ $\left(
x,y\right) \in \mathcal{X\times }\mathbb{R}$. Finally, for any $n\geq
n_{0}\left( u_{m},\alpha \right) $ we get that,%
\begin{equation}
\hat{\beta}_{m}-\beta _{m}=\left( I_{D}+A_{n,D}\right) ^{-1}F_{y,n}
\label{betahatminusbetam}
\end{equation}%
and 
\begin{equation}
\left\Vert \hat{s}_{m}-s_{m}\right\Vert _{2}^{2}=\left\vert \hat{\beta}%
_{m}-\beta _{m}\right\vert _{2}^{2}=\left\vert \left( I_{D}+A_{n,D}\right)
^{-1}F_{y,n}\right\vert _{2}^{2}  \label{from_s_to_beta}
\end{equation}%
By setting $B_{n,D}=\left( I_{D}+A_{n,D}\right) ^{-1}-I_{D}$, it thus holds,%
\begin{align}
\left\vert \left\Vert \hat{s}_{m}-s_{m}\right\Vert _{2}^{2}-\left\vert
F_{y,n}\right\vert _{2}^{2}\right\vert & =\left\vert \left\vert \left(
I_{D}+B_{n,D}\right) F_{y,n}\right\vert _{2}^{2}-\left\vert
F_{y,n}\right\vert _{2}^{2}\right\vert   \notag \\
& =\left\vert \left\vert B_{n,D}F_{y,n}\right\vert _{2}^{2}+2\left\langle
F_{y,n},B_{n,D}F_{y,n}\right\rangle \right\vert   \notag \\
& \leq \left( \left\vert \left\Vert B_{n,D}\right\Vert \right\vert
^{2}+2\left\vert \left\Vert B_{n,D}\right\Vert \right\vert \right)
\left\vert F_{y,n}\right\vert _{2}^{2}  \label{from_s_to_F}
\end{align}%
and for any $n\geq n_{0}\left( u_{m},\alpha \right) $, 
\begin{equation}
\left\vert \left\Vert B_{n,D}\right\Vert \right\vert \leq \frac{\left\vert
\left\Vert A_{n,D}\right\Vert \right\vert }{1-\left\vert \left\Vert
A_{n,D}\right\Vert \right\vert }\leq 2\left\vert \left\Vert
A_{n,D}\right\Vert \right\vert \leq L_{A_{-},u_{m},\alpha }\frac{D}{\sqrt{n}}%
\text{ .}  \label{majo_B}
\end{equation}%
Combining (\ref{from_s_to_beta}) and (\ref{majo_B}) implies that, for any $%
n\geq n_{0}\left( u_{m},\alpha \right) $,%
\begin{equation*}
\left\vert \left\Vert \hat{s}_{m}-s_{m}\right\Vert _{2}^{2}-\left\vert
F_{y,n}\right\vert _{2}^{2}\right\vert \leq L_{A_{-},u_{m},\alpha }\frac{D}{%
\sqrt{n}}\left\vert F_{y,n}\right\vert _{2}^{2}\text{ ,}
\end{equation*}%
and the proof simply follows by using Lemma \ref{Lemma_F_n_y}\ together with
the latter inequality.
\end{proof}

\bigskip

\begin{lemma}
\label{lemma_control_A_n_d}Recall that $A_{n,D}=\left( \left( P_{n}-P\right)
\left( \varphi _{k}\varphi _{l}\right) \right) _{k,l=1,..,D}$ is a $D\times D
$ matrix and that for a $D\times D$ matrix $A$, the operator norm $%
\left\vert \left\Vert 
\cdot%
\right\Vert \right\vert $ associated to the quadratic norm on the vectors is%
\begin{equation*}
\left\vert \left\Vert A\right\Vert \right\vert =\sup_{x\neq 0}\frac{%
\left\vert Ax\right\vert _{2}}{\left\vert x\right\vert _{2}}\text{ }.
\end{equation*}%
Then, under Assumption (\textbf{H3}), the following inequalities hold on an
event of probability at least $1-n^{-\alpha }$,%
\begin{equation}
\left\vert \left\Vert A_{n,D}\right\Vert \right\vert \leq L_{u_{m},\alpha }%
\frac{D}{\sqrt{n}}\left( 1+\sqrt{\frac{\ln n}{D}}\right) \leq \frac{1}{2}%
\text{ .}  \label{ineq_A_n_D}
\end{equation}
\end{lemma}

\begin{proof}
Let us denote $B_{1}$ the unit ball of $\left( m,\left\Vert \cdot
\right\Vert _{2}\right) $. It holds,%
\begin{eqnarray*}
\left\vert \left\Vert A_{n,D}\right\Vert \right\vert ^{2}
&=&\sup_{\left\vert x\right\vert _{2}=1}\left\vert A_{n,D}x\right\vert
_{2}^{2} \\
&=&\sup_{\left\vert x\right\vert _{2}=1}\sum_{k=1}^{D}\left(
\sum_{l=1}^{D}x_{l}\left( P_{n}-P\right) \left( \varphi _{k}\varphi
_{l}\right) \right) ^{2} \\
&=&\sup_{s\in B_{1}}\sum_{k=1}^{D}\left( \left( P_{n}-P\right) \left(
\varphi _{k}s\right) \right) ^{2} \\
&=&\sup_{s,t\in B_{1}}\left( \left( P_{n}-P\right) \left( s\cdot t\right)
\right) ^{2}\text{ .}
\end{eqnarray*}%
Hence,%
\begin{equation}
\left\vert \left\Vert A_{n,D}\right\Vert \right\vert =\sup_{s,t\in
B_{1}}\left( P_{n}-P\right) \left( st\right) \text{ .}
\label{norm_A_n_D_sup}
\end{equation}%
We will now apply Bousquet's concentration inequality (\ref{bousquet_2}) to
control the deviations of the supremum of the empirical process (\ref%
{norm_A_n_D_sup}). We have,%
\begin{eqnarray*}
\mathbb{E}\left[ \left\vert \left\Vert A_{n,D}\right\Vert \right\vert \right]
&\leq &\mathbb{E}^{1/2}\left[ \left\vert \left\Vert A_{n,D}\right\Vert
\right\vert ^{2}\right] \leq \mathbb{E}^{1/2}\left[ \sum_{k,l=1}^{D}\left(
P_{n}-P\right) ^{2}\left( \varphi _{k}\varphi _{l}\right) \right] \\
&\leq &\sqrt{\frac{\sum_{k,l=1}^{D}\mathbb{E}\left[ \varphi _{k}^{2}\varphi
_{l}^{2}\right] }{n}}\leq \frac{u_{m}D}{\sqrt{n}}\text{ ,}
\end{eqnarray*}%
where we used Assumption (\textbf{H3}) in the last inequality. Furthermore,
using (\textbf{H3}) and Remark \ref{remark_H4},%
\begin{equation*}
\sup_{s,t\in B_{1}}\mathbb{V}\left( st\right) \leq \sup_{s\in
B_{1}}\left\Vert s\right\Vert _{\infty }^{2}\leq u_{m}^{2}D\text{ \ \ and \
\ }\sup_{s,t\in B_{1}}\left\Vert st\right\Vert _{\infty }\leq \sup_{s\in
B_{1}}\left\Vert s\right\Vert _{\infty }^{2}\leq u_{m}^{2}D\text{ .}
\end{equation*}%
Hence, Bousquet's concentration inequality (\ref{bousquet_2}) gives (by
taking $\mathcal{F}=\left\{ st\text{ };\text{ }s,t\in B_{1}\right\} $ and $%
\varepsilon =1$), for any $x\geq 0$,%
\begin{equation*}
\mathbb{P}\left[ \left\vert \left\Vert A_{n,D}\right\Vert \right\vert \geq 
\frac{u_{m}D}{\sqrt{n}}+u_{m}\sqrt{\frac{2Dx}{n}}+\frac{u_{m}^{2}Dx}{3n}%
\right] \leq \exp \left( -x\right) \text{ .}
\end{equation*}%
Now, we get (\ref{ineq_A_n_D}) by taking $x=\alpha \ln n$ in the latter
inequality.
\end{proof}

\bigskip

\begin{lemma}
\label{Lemma_F_n_y}Let us denote $\psi _{m}\left( x,y\right) =y-s_{m}\left(
x\right) $. Assume that (\textbf{H1-3}) and recall that $F_{y,n}=\left(
\left( P_{n}-P\right) \left( \psi _{m}\varphi _{k}\right) \right)
_{k=1}^{D}\in \mathbb{R}^{D}$. Then%
\begin{equation}
\mathbb{P}\left( \left( 1-L_{A,A_{+},A_{-},u_{m},\left\Vert \sigma
\right\Vert _{2},\alpha }\sqrt{\frac{\ln n}{D}}\right) \frac{D}{n}\mathcal{C}%
_{m}^{2}\leq \left\Vert F_{y,n}\right\Vert _{2}^{2}\right) \geq 1-n^{-\alpha
}  \label{upper_F_sq}
\end{equation}%
and%
\begin{equation}
\mathbb{P}\left( \left\Vert F_{y,n}\right\Vert _{2}^{2}\leq \left(
1+L_{A,A_{+},A_{-},u_{m},\left\Vert \sigma \right\Vert _{2},\alpha }\sqrt{%
\frac{\ln n}{D}}\right) \frac{D}{n}\mathcal{C}_{m}^{2}\right) \geq
1-n^{-\alpha }\text{ ,}  \label{lower_F_norm_2}
\end{equation}%
where%
\begin{equation*}
\mathcal{C}_{m}^{2}=\frac{1}{D}\sum_{k=0}^{D-1}%
\var%
\left( \psi _{m}%
\cdot%
\varphi _{k}\right) \text{ .}
\end{equation*}
\end{lemma}

\begin{proof}
It holds%
\begin{equation*}
\left\Vert F_{y,n}\right\Vert _{2}=\sqrt{\sum_{k=1}^{D}\left( \left(
P_{n}-P\right) \left( \psi _{m}\varphi _{k}\right) \right) ^{2}}=\sup_{s\in
B_{1}}\left( P_{n}-P\right) \left( \psi _{m}s\right) \text{ .}
\end{equation*}%
We are thus reduced to the study of the supremum of an empirical process. We
have, by the hypotheses (\textbf{H1}), (\textbf{H3}) and Remark \ref%
{remark_H4},%
\begin{equation}
\sigma ^{2}:=\sup_{s\in B_{1}}%
\var%
\left( \psi _{m}s\right) \leq \left\Vert \psi _{m}\right\Vert _{\infty
}^{2}\leq 4A^{2}\text{ \ \ \ and \ \ }b:=\sup_{s\in B_{1}}\left\Vert \psi
_{m}s\right\Vert _{\infty }\leq 2Au_{m}\sqrt{D}\text{ .}  \label{upper_sig_b}
\end{equation}%
Furthermore, it holds%
\begin{equation*}
\mathbb{E}\left[ \left\Vert F_{y,n}\right\Vert _{2}^{2}\right] =\frac{D}{n}%
\mathcal{C}_{m}^{2}\text{ ,}
\end{equation*}%
which gives that for $\varkappa _{n}=2A\mathcal{C}_{m}^{-1}D^{-1/2}\max
\left\{ 1\text{ };\text{ }\sqrt{A_{+}}u_{m}\right\} $ , the two following
inequalities are satisfied,%
\begin{equation*}
\varkappa _{n}^{2}\mathbb{E}\left[ \left\Vert F_{y,n}\right\Vert _{2}^{2}%
\right] \geq \frac{\sigma ^{2}}{n}
\end{equation*}%
and%
\begin{equation*}
\varkappa _{n}^{2}\sqrt{\mathbb{E}\left[ \left\Vert F_{y,n}\right\Vert
_{2}^{2}\right] }\geq \frac{b}{n}\text{ .}
\end{equation*}%
Hence, by Theorem \ref{cor_alt_hoff_2} applied with $\mathcal{F}=B_{1}$, we
have 
\begin{equation}
\left( 1-\frac{L_{A,A_{+},u_{m},\left\Vert \sigma \right\Vert _{2}}}{\sqrt{D}%
}\right) \mathcal{C}_{m}\sqrt{\frac{D}{n}}\leq \mathbb{E}\left[ \left\Vert
F_{y,n}\right\Vert _{2}\right] \text{ }.  \label{lower_expect_F}
\end{equation}%
We also have%
\begin{equation}
\mathbb{E}\left[ \left\Vert F_{y,n}\right\Vert _{2}\right] \leq \sqrt{%
\mathbb{E}\left[ \left\Vert F_{y,n}\right\Vert _{2}^{2}\right] }=\mathcal{C}%
_{m}\sqrt{\frac{D}{n}}\text{ .}  \label{upper_expect_F}
\end{equation}%
Now, by combining the bounds obtained in (\ref{upper_expect_F}) and (\ref%
{lower_expect_F}) with Inequality (\ref{klein_rio_2}) applied with $\mathcal{%
F}=B_{1}$, $\varepsilon =n^{-1/4}\sqrt{\ln n}$ and $x=\alpha \ln n$, we get
that on an event of probability at least $1-n^{-\alpha }$,%
\begin{eqnarray*}
\left\Vert F_{y,n}\right\Vert _{2} &\geq &-\sqrt{\frac{2\sigma ^{2}\alpha
\ln n}{n}}+\left( 1-\varepsilon \right) \mathbb{E}\left[ \left\Vert
F_{y,n}\right\Vert _{2}\right] -\left( \frac{1}{\varepsilon }+1\right) \frac{%
b\alpha \ln n}{n} \\
&\geq &\left( 1-L_{A,A_{+},u_{m},\left\Vert \sigma \right\Vert _{2},\alpha }%
\sqrt{\frac{\ln n}{D}}\right) \sqrt{\frac{D}{n}}\mathcal{C}_{m}\text{ .}
\end{eqnarray*}%
Then easy calculations allow to derive Inequality (\ref{lower_F_norm_2})
from the latter lower bound.

Finally, combining the bounds obtained in (\ref{upper_expect_F}) and (\ref%
{lower_expect_F}) with Inequality (\ref{bousquet_2}) applied with $\mathcal{F%
}=B_{1}$, $\varepsilon =n^{-1/4}\sqrt{\ln n}$ and $x=\alpha \ln n$, we also
get that on an event of probability at least $1-n^{-\alpha }$,%
\begin{eqnarray*}
\left\Vert F_{y,n}\right\Vert _{2} &\leq &\sqrt{\frac{2\sigma ^{2}\alpha \ln
n}{n}}+\left( 1+\varepsilon \right) \mathbb{E}\left[ \left\Vert
F_{y,n}\right\Vert _{2}\right] +\left( \frac{1}{\varepsilon }+\frac{1}{3}%
\right) \frac{b\alpha \ln n}{n} \\
&\leq &\left( 1+L_{A,A_{+},u_{m},\left\Vert \sigma \right\Vert _{2},\alpha }%
\sqrt{\frac{\ln n}{D}}\right) \sqrt{\frac{D}{n}}\mathcal{C}_{m}\text{ ,}
\end{eqnarray*}%
which readily gives (\ref{upper_F_sq}).
\end{proof}

\subsubsection{Probabilistic Tools\label{section_probabilistic_tools}}

We recall here the main probabilistic results that are instrumental in the
proof of Theorem \ref{theorem_fourier}\ above.

\noindent Denote by 
\begin{equation*}
P_{n}=\frac{1}{n}\sum_{i=1}^{n}\delta _{\xi _{i}}
\end{equation*}%
the empirical measure associated to the sample $\left( \xi _{1},...,\xi
_{n}\right) $ and by 
\begin{equation*}
\left\Vert P_{n}-P\right\Vert _{\mathcal{F}}=\sup_{f\in \mathcal{F}%
}\left\vert \left( P_{n}-P\right) \left( f\right) \right\vert
\end{equation*}%
the supremum of the empirical process over $\mathcal{F}$.

\noindent We turn now to concentration inequalities for the empirical
process around its mean. Bousquet's inequality \cite{Bousquet:02} provides
optimal constants for the deviations at the right. Klein-Rio's inequality 
\cite{Klein_Rio:05} gives sharp constants for the deviations at the left,
that slightly improves Klein's inequality \cite{Klein:02}.

\begin{theorem}
\label{theorem_concentrations}Let $\left( \xi _{1},...,\xi _{n}\right) $ be $%
n$ i.i.d. random variables having common law $P$ and taking values in a
measurable space $\mathcal{Z}$. If $\mathcal{F}$ is a class of measurable
functions from $\mathcal{Z}$ to $\mathbb{R}$ satisfying%
\begin{equation*}
\text{\ }\left\vert f\left( \xi _{i}\right) -Pf\right\vert \leq b\text{ \ \ }%
a.s.,\text{ for all }f\in \mathcal{F},\text{ }i\leq n,
\end{equation*}%
then, by setting 
\begin{equation*}
\sigma _{\mathcal{F}}^{2}=\sup_{f\in \mathcal{F}}\left\{ P\left(
f^{2}\right) -\left( Pf\right) ^{2}\right\} ,
\end{equation*}%
we have,\ for all $x\geq 0$,

\textbf{Bousquet's inequality }:%
\begin{equation}
\mathbb{P}\left[ \left\Vert P_{n}-P\right\Vert _{\mathcal{F}}-\mathbb{E}%
\left[ \left\Vert P_{n}-P\right\Vert _{\mathcal{F}}\right] \geq \sqrt{%
2\left( \sigma_{\mathcal{F}}^{2}+2b\mathbb{E}\left[ \left\Vert
P_{n}-P\right\Vert _{\mathcal{F}}\right] \right) \frac{x}{n}}+\frac{bx}{3n}%
\right] \leq\exp\left( -x\right)  \label{bousquet}
\end{equation}
and we can deduce that, for all $\varepsilon,x>0$, it holds%
\begin{equation}
\mathbb{P}\left[ \left\Vert P_{n}-P\right\Vert _{\mathcal{F}}-\mathbb{E}%
\left[ \left\Vert P_{n}-P\right\Vert _{\mathcal{F}}\right] \geq\sqrt {%
2\sigma_{\mathcal{F}}^{2}\frac{x}{n}}+\varepsilon\mathbb{E}\left[ \left\Vert
P_{n}-P\right\Vert _{\mathcal{F}}\right] +\left( \frac{1}{\varepsilon}+\frac{%
1}{3}\right) \frac{bx}{n}\right] \leq\exp\left( -x\right) .
\label{bousquet_2}
\end{equation}

\textbf{Klein-Rio's inequality :}%
\begin{equation}
\mathbb{P}\left[ \mathbb{E}\left[ \left\Vert P_{n}-P\right\Vert _{\mathcal{F}%
}\right] -\left\Vert P_{n}-P\right\Vert _{\mathcal{F}}\geq \sqrt{2\left(
\sigma_{\mathcal{F}}^{2}+2b\mathbb{E}\left[ \left\Vert P_{n}-P\right\Vert _{%
\mathcal{F}}\right] \right) \frac{x}{n}}+\frac{bx}{n}\right] \leq\exp\left(
-x\right)  \label{klein_rio}
\end{equation}
and again, we can deduce that, for all $\varepsilon,x>0$, it holds%
\begin{equation}
\mathbb{P}\left[ \mathbb{E}\left[ \left\Vert P_{n}-P\right\Vert _{\mathcal{F}%
}\right] -\left\Vert P_{n}-P\right\Vert _{\mathcal{F}}\geq \sqrt{2\sigma_{%
\mathcal{F}}^{2}\frac{x}{n}}+\varepsilon\mathbb{E}\left[ \left\Vert
P_{n}-P\right\Vert _{\mathcal{F}}\right] +\left( \frac {1}{\varepsilon}%
+1\right) \frac{bx}{n}\right] \leq\exp\left( -x\right) .  \label{klein_rio_2}
\end{equation}
\end{theorem}

The following theorem is proved in \cite{saum:12}, Corollary 25. It can be
derived from a Theorem by Rio \cite{Rio2001}, improving on previous results
by Ledoux, and controlling the variance of the supremum of an empirical
process with bounded arguments (see also Theorem 11.10 in \cite%
{BouLugosiMassart:13}).

\begin{theorem}
\label{cor_alt_hoff_2}Under notations of Theorem \ref{theorem_concentrations}%
, if some $\varkappa _{n}\in \left( 0,1\right) $ exists such that%
\begin{equation*}
\varkappa _{n}^{2}\mathbb{E}\left[ \left\Vert P_{n}-P\right\Vert _{\mathcal{F%
}}^{2}\right] \geq \frac{\sigma ^{2}}{n}
\end{equation*}%
and%
\begin{equation*}
\varkappa _{n}^{2}\sqrt{\mathbb{E}\left[ \left\Vert P_{n}-P\right\Vert _{%
\mathcal{F}}^{2}\right] }\geq \frac{b}{n}
\end{equation*}%
then we have, for a numerical constant $A_{1,-}$, 
\begin{equation*}
\left( 1-\varkappa _{n}A_{1,-}\right) \sqrt{\mathbb{E}\left[ \left\Vert
P_{n}-P\right\Vert _{\mathcal{F}}^{2}\right] }\leq \mathbb{E}\left[
\left\Vert P_{n}-P\right\Vert _{\mathcal{F}}\right] \text{ }.
\end{equation*}
\end{theorem}

\subsection{Proofs related to Section \protect\ref%
{ssection_refined_small_ball}\label{section_proof_refined_small_ball}}

\begin{proof}[Proof of Proposition \protect\ref{prop_small_ball_Gamma}]
Take $f\in \Lambda _{\nu }\left( L_{1},L_{2}\right) $ and $\kappa _{0}\in
\left( 0,1\right) $. Then,%
\begin{eqnarray}
\left\Vert f\right\Vert _{\infty } &\leq &\sqrt{2}\sum_{k\in \mathbb{N}%
_{\ast }}\left\vert \beta _{k}\left( f\right) \right\vert  \notag \\
&\leq &\sqrt{2}L_{1}\sqrt{\sum_{k\in \mathbb{N}_{\ast }}\frac{\left\vert
\beta _{k}\left( f\right) \right\vert }{k^{\nu }}}\text{ ,}
\label{ineq_1_unif}
\end{eqnarray}%
where the second inequality follows from Cauchy-Schwarz inequality.
Furthermore, by Cauchy-Schwarz inequality again,%
\begin{equation}
\sum_{k\in \mathbb{N}_{\ast }}\frac{\left\vert \beta _{k}\left( f\right)
\right\vert }{k^{\nu }}\leq \left( \sum_{k\in \mathbb{N}_{\ast }}\frac{1}{%
k^{2\nu }}\right) ^{1/2}\left( \sum_{k\in \mathbb{N}_{\ast }}\beta
_{k}^{2}\left( f\right) \right) ^{1/2}=\sqrt{C_{\nu }}\left\Vert
f\right\Vert _{2}\text{ ,}  \label{ineq_2_unif}
\end{equation}%
with $C_{\nu }:=\sum_{k\in \mathbb{N}_{\ast }}k^{-2\nu }<+\infty $ since $%
\nu >1/2$. Combining (\ref{ineq_1_unif}), (\ref{ineq_2_unif}) and the fact
that $\left\Vert f\right\Vert _{\infty }\geq L_{2}>0$, we get%
\begin{equation*}
\left\Vert f\right\Vert _{\infty }\leq \frac{\left\Vert f\right\Vert
_{\infty }^{2}}{L_{2}}\leq 2C_{\nu }\frac{L_{1}^{2}}{L_{2}}\left\Vert
f\right\Vert _{2}\text{ .}
\end{equation*}%
The conclusion then follows from Paley-Zygmund's inequality (Corollary 3.3.2
in \cite{delaPenaGine:99}), since it holds%
\begin{equation*}
\mathbb{P}\left( \left\vert f\left( X\right) \right\vert \geq \kappa
_{0}\left\Vert f\right\Vert _{2}\right) \mathbb{\geq }\left( 1-\kappa
_{0}^{2}\right) \frac{\left\Vert f\right\Vert _{2}^{2}}{\left\Vert
f\right\Vert _{\infty }^{2}}\geq \frac{\left( 1-\kappa _{0}^{2}\right) }{%
4C_{\nu }^{2}}\frac{L_{2}^{2}}{L_{1}^{4}}>0\text{ .}
\end{equation*}
\end{proof}

We turn now to the proof of Theorem \ref{theorem_small_ball_modif_Fourier}.
The idea is to localize the calculations on a subset of the model $m$,
containing the estimator $\hat{s}_{m}$ w.h.p. and achieving the small-ball
condition with some absolute constants. Therefore, we first need the
following result, which is a direct extension of Theorem A in \cite%
{Lecue_Mendelson:15}.

\begin{theorem}
\label{theorem_ext_lecMendel}Let $S=\left\{ s_{1},...,s_{D}\right\} \subset
L_{2}\left( P^{X}\right) $ be a dictionary. Assume that a set $m_{0}\subset
m:=%
\Span%
\left( S\right) $ satisfies the small-ball condition with constants $\kappa
_{0}$ and $\beta _{0}$ (see Definition \ref{def_small_ball}\ above) and
contains, on an event $\Omega _{0}$, the least-squares estimator $\hat{s}%
_{m} $ on $m$, defined in (\ref{def_ERM}). Let $n\geq \left( 400\right)
^{2}D/\beta _{0}^{2}$ and set $\zeta =Y-s_{m}\left( X\right) $, where $s_{m}$
is the projection of the regression function $s_{\ast }$ onto $m$. Assume
further that one of the following two conditions holds:

\begin{enumerate}
\item $\zeta $ is independent of $X$ and $\mathbb{E}\zeta ^{2}\leq \sigma
^{2}$, or

\item $\left\vert \zeta \right\vert \leq \sigma $ almost surely.
\end{enumerate}

Then the estimator $\hat{s}_{m}$ satisfies for every $x>0$, with probability
at least $1-\mathbb{P}\left( \Omega _{0}^{c}\right) -\exp \left( -\beta
_{0}^{2}n/4\right) -\left( 1/x\right) $,%
\begin{equation*}
\left\Vert \hat{s}_{m}-s_{m}\right\Vert _{2}^{2}\leq \left( \frac{16}{\beta
_{0}\kappa _{0}^{2}}\right) ^{2}\frac{\sigma ^{2}Dx}{n}\text{ .}
\end{equation*}
\end{theorem}

Theorem \ref{theorem_ext_lecMendel} ensures that if an information is
available w.h.p. on the location of the estimator on the model $m$, then it
may be used to derive better rates by taking advantage of better small-ball
constants achieved on the restricted set containing the estimator.

The proof is omitted, since a careful reading of the proof of Theorem A in 
\cite{Lecue_Mendelson:15} allows to conclude that using a localization of
the estimator $\hat{s}_{m}$ does not change the reasoning, neither the
validity of the arguments.

The following proposition states that indeed, when the regression function $%
s_{\ast }$ is sufficiently regular, then so is the least-squares estimator
on the first elements of the Fourier basis.

\begin{proposition}
\label{prop_regu_est}Take $v,L_{1},L_{2},z>0$ and assume that $s_{\ast }\in
\Lambda _{\nu }\left( L_{1},L_{2}\right) $. For a dimension $D$ satisfying 
\begin{equation*}
0<\left( 2\sqrt{2}L_{1}L_{2}^{-1}\right) ^{1/\nu }\leq D\leq L_{\nu }\left(
n/\ln n\right) ^{\frac{1}{2\left( \nu +1\right) }}
\end{equation*}%
and for $z\leq L_{L_{1},L_{2},\sigma ,\nu }n/D^{2\left( \nu +1\right) }$, it
holds%
\begin{equation*}
\mathbb{P}\left( \hat{s}_{m}\in \Lambda _{\nu }\left( 2L_{1},\frac{L_{2}}{4}%
\right) \right) \geq 1-n^{-2}-1/z\text{ .}
\end{equation*}
\end{proposition}

We are now in a position to prove Theorem \ref%
{theorem_small_ball_modif_Fourier}. The proof of Proposition \ref%
{prop_regu_est} is thus postponed after the proof of Theorem \ref%
{theorem_small_ball_modif_Fourier}.\bigskip

\begin{proof}[Proof of Theorem \protect\ref{theorem_small_ball_modif_Fourier}%
]
Apply Theorem \ref{theorem_ext_lecMendel} with 
\begin{equation*}
\Omega _{0}=\left\{ \hat{s}_{m}\in \Lambda _{\nu }\left( 2L_{1},\frac{L_{2}}{%
4}\right) \right\}
\end{equation*}%
and $x=z$. Then Proposition \ref{prop_small_ball_Gamma} ensures that\ on $%
\Omega _{0}$ the small-ball is achieved with parameters $\kappa
_{0}=2^{-1/2} $ and $\beta _{0}=C_{\nu }^{-2}L_{2}^{2}L_{1}^{-4}/8$. Hence,
the condition $n\geq \left( 400\right) ^{2}D/\beta _{0}^{2}$ is satisfied
for $D\leq L_{\nu }\left( n/\ln n\right) ^{\frac{1}{2\left( \nu +1\right) }}$
whenever $n\geq n_{0}\left( \nu ,L_{1},L_{2}\right) $. Theorem \ref%
{theorem_small_ball_modif_Fourier} then follows from Proposition \ref%
{prop_regu_est} and\ straightforward computations.
\end{proof}

\bigskip

Before proving Proposition \ref{prop_regu_est}, let us denote, for any $\nu
>0$, 
\begin{equation*}
\Lambda _{\nu }=\dbigcup_{L_{1},L_{2}>0}\Lambda _{\nu }\left(
L_{1},L_{2}\right) =\left\{ f=\sum_{k\geq 1}\beta _{k}\varphi _{k}\text{ };%
\text{ }\sum_{k\geq 1}k^{\nu }\left\vert \beta _{k}\right\vert <+\infty
\right\} \text{ .}
\end{equation*}%
For any $f\in \Lambda _{\nu }$, let us write $\left\Vert f\right\Vert
_{\Lambda ,\nu }=\sum_{k\geq 1}k^{\nu }\left\vert \left\langle f,\varphi
_{k}\right\rangle \right\vert $. It is easily seen that $\left\Vert \cdot
\right\Vert _{\Lambda ,\nu }$ is a norm on the space $\Lambda _{\nu }$.

For a sequence $\beta =\left( \beta _{k}\right) _{k\geq 1}\in \mathbb{R}^{%
\mathbb{N}}$, we denote $\left\vert \beta \right\vert _{\Lambda ,\nu
}=\sum_{k\geq 1}k^{\nu }\left\vert \beta _{k}\right\vert \in \mathbb{R}%
_{+}\cup \left\{ +\infty \right\} $ and $\tilde{\Lambda}_{\nu }:=\left\{
\beta =\left( \beta _{k}\right) _{k\geq 1}\in \mathbb{R}^{\mathbb{N}}\text{ }%
;\text{ }\sum_{k\geq 1}k^{\nu }\left\vert \beta _{k}\right\vert <+\infty
\right\} $. Furthermore, for a $D\times D$ matrix $A$, the operator norm $%
\left\vert \left\Vert 
\cdot%
\right\Vert \right\vert _{\Lambda ,\nu }$ associated to the norm $\left\vert
\cdot \right\vert _{\Lambda ,\nu }$ on the vectors (seen as sequences with
finite support) is%
\begin{equation*}
\left\vert \left\Vert A\right\Vert \right\vert _{\Lambda ,\nu }:=\sup_{x\in 
\mathbb{R}^{D},x\neq 0}\frac{\left\vert Ax\right\vert _{\Lambda ,\nu }}{%
\left\vert x\right\vert _{\Lambda ,\nu }}\text{ }.
\end{equation*}%
By simple computations it holds, for any matrix $A=\left( A_{k,l}\right)
_{1\leq k,l\leq D}$,%
\begin{eqnarray}
\left\vert \left\Vert A\right\Vert \right\vert _{\Lambda ,\nu } &=&\sup
\left\{ \sum_{k=1}k^{\nu }\left\vert \sum_{l=1}^{D}A_{k,l}x_{l}\right\vert 
\text{ };\text{ }x\in \mathbb{R}^{D}\And \sum_{k=1}^{D}k^{\nu }\left\vert
x_{k}\right\vert =1\right\}  \notag \\
&=&\sum_{k=1}k^{\nu }\max_{l=1,...,D}\left\vert \frac{A_{k,l}}{l^{\nu }}%
\right\vert \text{ .}  \label{norm_mat_nu}
\end{eqnarray}

\bigskip

\begin{proof}[Proof of Proposition \protect\ref{prop_regu_est}]
Let us write $s_{\ast }=\sum_{k\geq 1}\beta _{k}\varphi _{k}$. Thus $%
s_{m}=\sum_{k=1}^{D}\beta _{k}\varphi _{k}$ and since $s_{\ast }\in \Lambda
_{\nu }\left( L_{1},L_{2}\right) $, it holds $\sum_{k\geq 1}k^{\nu
}\left\vert \beta _{k}\right\vert \leq L_{1}$. Hence, it holds in particular 
$\sum_{k=1}^{D}k^{\nu }\left\vert \beta _{k}\right\vert \leq L_{1}$ and 
\begin{equation*}
\left\Vert s_{\ast }-s_{m}\right\Vert _{\infty }\leq \sqrt{2}\sum_{k\geq
D+1}\left\vert \beta _{k}\right\vert \leq \frac{\sqrt{2}}{D^{\nu }}%
\sum_{k\geq D+1}k^{\nu }\left\vert \beta _{k}\right\vert \leq \frac{\sqrt{2}%
L_{1}}{D^{\nu }}\text{ .}
\end{equation*}%
Consequently, we have $\ \left\Vert s_{\ast }-s_{m}\right\Vert _{\infty
}\leq L_{2}/2$ and so $\left\Vert s_{m}\right\Vert _{\infty }\geq \left\Vert
s_{\ast }\right\Vert _{\infty }-\left\Vert s_{\ast }-s_{m}\right\Vert
_{\infty }\geq L_{2}/2$ whenever $D\geq \left( 2\sqrt{2}L_{1}L_{2}^{-1}%
\right) ^{1/\nu }$. Therefore, for such dimension $D$, we get $s_{m}\in
\Lambda _{\nu }\left( L_{1},L_{2}/2\right) $.

Now, we write $\hat{s}_{m}=\sum_{k=1}^{D}\hat{\beta}_{k}\varphi _{k},$ $\hat{%
\beta}_{m}=\left( \hat{\beta}_{k}\right) _{k=1}^{D}$ and the define the
following set, 
\begin{equation*}
\Omega _{\Lambda }=\left\{ \left\vert \left\Vert A_{n,D}\right\Vert
\right\vert _{\Lambda ,\nu }\leq \frac{4\left( D+1\right) ^{\nu +1}}{\nu +1}%
\sqrt{\frac{3\ln n}{n}}\leq \frac{1}{2}\right\} \dbigcap \left\{ \left\vert
F_{y,n}\right\vert _{\Lambda ,\nu }\leq \frac{\left( D+1\right) ^{\nu +1}}{%
\nu +1}\sqrt{\frac{2\sigma ^{2}z}{n}}\right\} \text{ ,}
\end{equation*}%
where the matrix $A_{n,D}$ and the vector $F_{y,n}$ are defined respectively
in (\ref{system_Fourier}) and (\ref{def_F_y_n}), where $\left( \varphi
_{k}\right) _{k=1}^{D}$ should stand for the first $D$ elements of the
Fourier basis this time. On $\Omega _{\Lambda }$, the matrix $Id+A_{n,D}\ $%
is invertible and it holds$\ \hat{\beta}_{m}-\beta _{m}=\left(
I_{D}+A_{n,D}\right) ^{-1}F_{y,n}$.

From Lemmas \ref{lemma_norm_Gamma} and \ref{Lemma_F_n_y_Gamma}, there exists
an integer $n_{0}\left( \nu \right) $ such that for any $n\geq n_{0}\left(
\nu \right) $, $\mathbb{P}\left( \Omega _{\Lambda }\right) \geq 1-n^{-2}-1/z$%
. Furthermore, on $\Omega _{\Lambda }$, we have,%
\begin{gather}
\left\Vert \hat{s}_{m}-s_{m}\right\Vert _{\Lambda ,\nu }=\left\vert \hat{%
\beta}_{m}-\beta _{m}\right\vert _{\Lambda ,\nu }\leq \left\vert \left\Vert
\left( I_{D}+A_{n,D}\right) ^{-1}\right\Vert \right\vert \left\vert
F_{y,n}\right\vert _{\Lambda ,\nu }  \notag \\
\leq \left( 1+2\left\vert \left\Vert A_{n,D}\right\Vert \right\vert
_{\Lambda ,\nu }\right) \left\vert F_{y,n}\right\vert _{\Lambda ,\nu }\leq 
\frac{2\left( D+1\right) ^{\nu +1}}{\nu +1}\sqrt{\frac{\sigma ^{2}z}{n}}
\label{majo_norm_Gamma}
\end{gather}%
and%
\begin{equation}
\left\Vert \hat{s}_{m}-s_{m}\right\Vert _{\infty }\leq \sqrt{2}%
\sum_{k=1}^{D}\left\vert \hat{\beta}_{k}-\beta _{k}\right\vert \leq \sqrt{2}%
\left\vert \hat{\beta}_{m}-\beta _{m}\right\vert _{\Lambda ,\nu }\leq \frac{2%
\sqrt{2}\left( D+1\right) ^{\nu +1}}{\nu +1}\sqrt{\frac{\sigma ^{2}z}{n}}%
\text{ .}  \label{majo_norm_sup}
\end{equation}%
Finally, it is easily seen from (\ref{majo_norm_Gamma}) and (\ref%
{majo_norm_sup}) that there exists a constant $L_{L_{1},L_{2},\sigma ,\nu }$
such that if $z\leq L_{L_{1},L_{2},\sigma ,\nu }n/D^{2\left( \nu +1\right) }$%
, then $\left\Vert \hat{s}_{m}-s_{m}\right\Vert _{\Lambda ,\nu }\leq L_{1}$
and $\left\Vert \hat{s}_{m}-s_{m}\right\Vert _{\infty }\leq L_{2}/4$.
\end{proof}

\bigskip

\bigskip

\begin{lemma}
\label{lemma_norm_Gamma}Recall that $A_{n,D}=\left( \left( P_{n}-P\right)
\left( \varphi _{k}\varphi _{l}\right) \right) _{k,l=1,..,D}$ is a $D\times
D $ matrix. Then the following inequalities hold on an event of probability
at least $1-D^{2}n^{-\alpha }$,%
\begin{equation}
\left\vert \left\Vert A_{n,D}\right\Vert \right\vert _{\Lambda ,\nu }\leq 
\frac{2\left( D+1\right) ^{\nu +1}}{\nu +1}\sqrt{\frac{\alpha \ln n}{n}}%
\left( 1+\sqrt{\frac{\alpha \ln n}{n}}\right) \text{ .}
\label{majo_norm_Ad_gene}
\end{equation}%
Consequently, there exists a constant $L_{\nu }>0$ such that for $D\leq
L_{\nu }\left( n/\ln n\right) ^{\frac{1}{2\left( \nu +1\right) }}$, it holds
for any $n\geq n_{0}\left( \nu \right) $, with probability at least $%
1-n^{-2} $,%
\begin{equation}
\left\vert \left\Vert A_{n,D}\right\Vert \right\vert _{\Lambda ,\nu }\leq 
\frac{4\left( D+1\right) ^{\nu +1}}{\nu +1}\sqrt{\frac{3\ln n}{n}}\leq \frac{%
1}{2}\text{ .}  \label{majo_norm_Ad_part}
\end{equation}
\end{lemma}

\begin{proof}
By (\ref{norm_mat_nu}) we have,%
\begin{equation}
\left\vert \left\Vert A_{n,D}\right\Vert \right\vert _{\Lambda ,\nu
}=\sum_{k=1}^{D}k^{\nu }\max_{l=1,...,D}\left\vert \frac{\left(
P_{n}-P\right) \left( \varphi _{k}\varphi _{l}\right) }{l^{\nu }}\right\vert 
\text{ .}  \label{norm_AnD_1}
\end{equation}%
Furthermore, for any $k,l=1,...,D$, it holds%
\begin{equation*}
\mathbb{V}\left( \varphi _{k}\varphi _{l}\right) \leq \left\Vert \varphi
_{k}\right\Vert _{\infty }^{2}\mathbb{E}\left[ \varphi _{l}^{2}\right] \leq 2%
\text{ \ \ and \ \ }\left\Vert \varphi _{k}\varphi _{l}\right\Vert _{\infty
}\leq 2\text{ .}
\end{equation*}%
Hence, for any $x>0$, we get by Bernstein's inequality (see for instance 
\cite{Massart:07}), that on an event $\Omega _{k,l}\left( x\right) $ of
probability at least $1-2\exp \left( -x\right) $,%
\begin{equation*}
\left\vert \left( P_{n}-P\right) \left( \varphi _{k}\varphi _{l}\right)
\right\vert \leq 2\sqrt{\frac{x}{n}}+\frac{2x}{n}\text{ .}
\end{equation*}%
Then Identity (\ref{norm_AnD_1}) implies that, for any $\alpha >0$, on the
event $\Omega _{D}=\dbigcap_{1\leq l\leq k\leq D}\Omega _{k,l}\left( \alpha
\ln n\right) $ of probability greater than $1-D^{2}/n^{\alpha }$,%
\begin{eqnarray}
\left\vert \left\Vert A_{n,D}\right\Vert \right\vert _{\Lambda ,\nu } &\leq
&2\left( \sqrt{\frac{\alpha \ln n}{n}}+\frac{\alpha \ln n}{n}\right)
\sum_{k=1}^{D}k^{\nu }  \notag \\
&\leq &\frac{2\left( D+1\right) ^{\nu +1}}{\nu +1}\sqrt{\frac{\alpha \ln n}{n%
}}\left( 1+\sqrt{\frac{\alpha \ln n}{n}}\right) \text{ .}
\label{majo_gene_proof}
\end{eqnarray}%
Thus (\ref{majo_norm_Ad_gene}) is proved and Inequality (\ref%
{majo_norm_Ad_part}) can be deduced from it by simply taking $\alpha =3$.
\end{proof}

\bigskip

\begin{lemma}
\label{Lemma_F_n_y_Gamma}Let us denote $\psi _{m}\left( x,y\right)
=y-s_{m}\left( x\right) $. Recall that $F_{y,n}=\left( \left( P_{n}-P\right)
\left( \psi _{m}\varphi _{k}\right) \right) _{k=1}^{D}\in \mathbb{R}^{D}$.
Then, for any $z>0$,%
\begin{equation}
\mathbb{P}\left( \left\vert F_{y,n}\right\vert _{\Lambda ,\nu }\leq \frac{%
\left( D+1\right) ^{\nu +1}}{\nu +1}\sqrt{\frac{2\sigma ^{2}z}{n}}\right)
\geq 1-\frac{1}{z}\text{ .}  \label{lower_F_Gam}
\end{equation}
\end{lemma}

\begin{proof}
It holds%
\begin{eqnarray*}
\sqrt{\mathbb{E}\left[ \left\vert F_{y,n}\right\vert _{\Lambda ,\nu }^{2}%
\right] } &=&\sqrt{\mathbb{E}\left[ \left( \sum_{k=1}^{D}k^{\nu }\left\vert
\left( P_{n}-P\right) \left( \psi _{m}\varphi _{k}\right) \right\vert
\right) ^{2}\right] } \\
&\leq &\sum_{k=1}^{D}k^{\nu }\sqrt{\mathbb{E}\left[ \left( P_{n}-P\right)
^{2}\left( \psi _{m}\varphi _{k}\right) \right] } \\
&\leq &\max_{k=1,...,D}\sqrt{\frac{\mathbb{E}\left[ \psi _{m}^{2}\varphi
_{k}^{2}\right] }{n}}\sum_{k=1}^{D}k^{\nu } \\
&\leq &\frac{\left( D+1\right) ^{\nu +1}}{\nu +1}\sqrt{\frac{2\sigma ^{2}}{n}%
}\text{ .}
\end{eqnarray*}%
Then Lemma \ref{Lemma_F_n_y_Gamma} follows from Markov's inequality.
\end{proof}

\bibliographystyle{alpha}
\bibliography{Slope_heuristics_regression_13}

\end{document}